\documentclass[11pt,a4paper]{article}
  \usepackage{latexsym,amssymb,amsfonts,amsmath,epsfig,tabularx,stmaryrd,amsthm}
  \usepackage[usenames]{xcolor}

    \newtheorem{theorem}{Theorem}[section]
  {\theoremstyle{definition}
  
  \newtheorem{remark}[theorem]{Remark}
}
  \newtheorem{lemma}[theorem]{Lemma}
  \newtheorem{corollary}[theorem]{Corollary}

  \numberwithin{equation}{section}

\usepackage{graphicx}
\usepackage{bbm}
\usepackage{pgfplots}
\usepackage{url}
\usepackage{algorithm}
\usepackage{algpseudocode}
\usepackage{subcaption}
\usepackage{mathtools}

\graphicspath{{.}{./Figures/}}

\usepackage{pdfsync}
\definecolor{indigo}{RGB}{161,11,112}
\definecolor{darkgreen}{RGB}{0,100,0}
\definecolor{darkred}{RGB}{139,0,0}
\definecolor{darkblue}{RGB}{0,0,190}
\definecolor{darkmagenta}{RGB}{180,0,180} 
\definecolor{darkorange}{RGB}{180,60,0} 
\definecolor{gray}{rgb}{0.66, 0.66, 0.66}

\newcommand{\bsa}{{\boldsymbol{a}}}

\newcommand{\bsc}{{\boldsymbol{c}}}

\newcommand{\bse}{{\boldsymbol{e}}}

\newcommand{\bsh}{{\boldsymbol{h}}}

\newcommand{\bst}{{\boldsymbol{t}}}

\newcommand{\bsv}{{\boldsymbol{v}}}
\newcommand{\bsw}{{\boldsymbol{w}}}
\newcommand{\bsx}{{\boldsymbol{x}}}
\newcommand{\bsy}{{\boldsymbol{y}}}
\newcommand{\bsz}{{\boldsymbol{z}}}

\newcommand{\bsF}{{\boldsymbol{F}}}

\newcommand{\bsL}{{\boldsymbol{L}}}

\newcommand{\bsX}{{\boldsymbol{X}}}

\newcommand{\bsell}{{\boldsymbol{\ell}}}
\newcommand{\bszero}{{\boldsymbol{0}}} 

\newcommand{\bsbeta}{{\boldsymbol{\beta}}}
\newcommand{\bsgamma}{{\boldsymbol{\gamma}}}

\newcommand{\bseta}{{\boldsymbol{\eta}}}

\newcommand{\bskappa}{{\boldsymbol{\kappa}}}
\newcommand{\bslambda}{{\boldsymbol{\lambda}}}
\newcommand{\bsmu}{{\boldsymbol{\mu}}}
\newcommand{\bsnu}{{\boldsymbol{\nu}}}

\newcommand{\bspsi}{{\boldsymbol{\psi}}}

\newcommand{\bsPhi}{{\boldsymbol{\Phi}}}

\newcommand{\calA}{{\mathcal{A}}}
\newcommand{\calB}{{\mathcal{B}}}

\newcommand{\calK}{{\mathcal{K}}}

\newcommand{\calO}{{\mathcal{O}}}

\newcommand{\calS}{{\mathcal{S}}}

\newcommand{\N}{{\mathbb{N}}} 
\newcommand{\R}{{\mathbb{R}}} 

\newcommand{\bbR}{{\mathbb{R}}} 
\newcommand{\bbZ}{{\mathbb{Z}}} 
\newcommand{\bbN}{{\mathbb{N}}}

\newcommand{\ri}{{\mathrm{i}}}
\newcommand{\rd}{{\mathrm{d}}}

\newcommand{\setu}{\mathfrak{u}}
\newcommand{\supp}{\mathrm{supp}}

\makeatletter
\let\@fnsymbol\@arabic
\makeatother
  \title{A higher order lattice-based method for high-dimensional numerical integration using periodization}
   \author{Felix Bartel\footnotemark[1] \and Alexander D.~Gilbert\footnotemark[2] \and Michael Griebel\footnotemark[3] \and Frances Y.~Kuo\footnotemark[2] \and Ian H.~Sloan\footnotemark[2]}

\begin{document}

\footnotetext[1]{Mathematical Institute for Machine Learning and Data Science (MIDS), Catholic
University of Eichst\"{a}tt--Ingolstadt, 85049 Ingolstadt, Germany.
\texttt{felix.bartel@ku.de}}
\footnotetext[2]{School of Mathematics and Statistics, UNSW Sydney, Sydney NSW 2052, Australia.\\
                           \texttt{alexander.gilbert@unsw.edu.au},\;
                           \texttt{f.kuo@unsw.edu.au},\;
                           \texttt{i.sloan@unsw.edu.au}
                           }
\footnotetext[3]{Institute for Numerical Simulation, University of Bonn, 53115 Bonn, Germany
and Fraunhofer Institute for Algorithms and Scientific Computing (SCAI), 53757 Sankt Augustin,
Germany.
\texttt{griebel@ins.uni-bonn.de}}

\maketitle

\begin{abstract}
We present a novel numerical integration method for non-periodic functions over the high-dimensional unit cube $[0,1]^s$ by a specially crafted (unequally-)\emph{weighted} quadrature rule using \emph{transformed} \emph{lattice} points, with the additional option of \emph{subsampling}. The method is designed for integrands whose mixed derivatives up to order $\alpha$ are square-integrable with respect to a product Chebyshev density. The ingredients are (i) periodizing the integrand, (ii) approximating the smooth periodic function by a kernel interpolant at lattice points, (iii) integrating exactly the product of the kernel interpolant and the nonsmooth density, together with subsampling alternatives for (ii) and~(iii). With $|J|$ corresponding to the number of function evaluations, we achieve the convergence rate close to the order $|J|^{-(\alpha-1/2)}$ and $|J|^{-\alpha}\,($1.110721$)^s$, where the implied constants are independent of $s$ under favorable conditions. We can interpolate these two results to trade between the convergence rate and the growth in $s$. We provide numerical experiments to demonstrate our theory.
\end{abstract}

\section{Introduction}

In this paper we present a novel approach to the numerical computation of integrals of non-periodic functions over the $s$-dimensional unit cube $[0,1]^s$, with $s$ large, by a \emph{specially crafted weighted quadrature rule using transformed (and optionally subsampled) lattice points}:
\begin{align*} 
  I(f) \coloneqq \int_{[0,1]^s} f(\bsx) \,\rd\bsx 
  \;\approx\; \sum_{k\in J} w_k\,f(\bspsi(\bst_k))
  \eqqcolon Q^{\bsX_J}_\bspsi(f).
\end{align*} 
The ingredients are firstly a simple coordinate transformation $\bsx = \bspsi(\bsy)$ that transforms the integrand to a smooth periodic function $F(\bsy) \coloneqq f(\bspsi(\bsy))$ multiplied by a non-smooth density; then approximation of the smooth function $F$ by a kernel approximation that uses the points $\bst_k$ of a carefully designed lattice; then exact integration of kernel interpolant times density to determine the weights $w_k$; with the optional alternative of replacing the full lattice by a randomly chosen subset $\bsX_J$.

The method is designed for functions $f$ that are somewhat smooth, having mixed derivatives up to order $\alpha \in \mathbb{N}$ that are well behaved at the cube boundaries, in the sense that 
\begin{equation*} \label{eq:fcondition}
  \int_{[0, 1]^s} \Big|\frac{\partial^\bsnu }{\partial \bsx^\bsnu} f(\bsx) \Big|^2\,
 \prod_{j=1}^s \frac{1}{\pi \sqrt{x_j(1-x_j)}} \,\rd \bsx < \infty
 \quad\mbox{for all}\quad \bsnu \in \{0, \ldots, \alpha\}^s, 
\end{equation*}
which is satisfied if for instance all the partial derivatives of $f$ are bounded on the unit cube.

\subsection{Background}

It is well known that, for functions in the classical Sobolev space $H_\alpha([0,1]^s)$ with smoothness $\alpha\ge 1$, a deterministic product quadrature with $N$ points has an error of the order $\calO(N^{-\alpha/s})$, thus one encounters the {\em curse of dimensionality}. Monte Carlo quadrature has a root-mean-square error of $\calO(N^{-1/2})$, which is independent of $s$ but suffers a slow rate regardless of $\alpha$. 

To obtain improved convergence, one should consider the smaller Sobolev space $H_\alpha^{\rm mix}([0,1]^s)$ of functions with bounded {\em mixed} derivatives of order $\alpha$. 
Conventional Quasi-Monte Carlo (QMC) methods \cite{Nie92,SJ94} use well chosen deterministic points to improve the rate close to $\calO(N^{-1})$. Lattice rules \cite{DKS13,DKP22} can achieve close to $\calO(N^{-\alpha})$ for periodic integrands, while the same rate can be achieved by higher order digital nets \cite{DP10} for general non-periodic integrands. The same main rate is also achieved by sparse grid quadrature \cite{BG04}.
However, all these bounds include a logarithmic factor in~$N$ that grows exponentially with~$s$, and often also exponentially with $\alpha$. For example, for sparse grids we encounter $(\log N)^{(\alpha+1)(s-1)}$, 
and for higher order digital nets we encounter $(\log N)^{\alpha \cdot s}$ or $(\log N)^{\alpha + s}$. The {\em optimal} exponent here is $(\log N)^{(s-1)/2}$, which is achieved by Frolov quadrature~\cite{UU16}, but the curse of dimensionality is still present in the $\log N$ term. Moreover, the implied constants all depend on $\alpha$ and exponentially on $s$, limiting these methods to moderate dimensions~$s$ in practice.

To lift this curse, one may work with a {\em weighted} Sobolev space $H^{\rm mix}_{\alpha,\bsgamma}([0,1]^s)$ and associated weighted norms, see \cite{SW98,DKS13}, where the underlying weight sequence ($\bsgamma$ introduced below) satisfies a certain summability condition. Then the aforementioned lattice rules and higher order QMC will converge with rate $\calO(N^{-\alpha+\delta})$ for arbitrarily small $\delta>0$, without any $\log N$ term, and with the implied constant \emph{independent of~$s$}. (This is the essence of ``\emph{strong tractability}'' \cite{NW08}.)
Similar results hold for anisotropically adapted sparse grid quadrature. 

There are other related strategies such as periodization, tent transform, and cosine transform, and we will discuss them briefly at the end of this paper. For now it is worth pointing out that the implied constants in the error bounds for different methods with the same convergence rate can depend on $\alpha$ in very different ways; some may grow super exponentially in $\alpha$, thus requiring a huge number of points in practice to resolve the preasymptotic regime for large~$\alpha$.

Our approach in this paper offers a different alternative to higher order QMC for the weighted Sobolev space. We construct, in both theory and practice, non-QMC, i.e., unequally weighted, quadrature rules which aim to have at the same time higher order convergence with respect to the number of function evaluations, and little or no growth with respect to~$s$. Specifically, with $|J|$ denoting the number of subsampled lattice points, we achieve a convergence rate close to the order 
\[
  |J|^{- (\alpha - 1/2)}
  \quad\mbox{and}\quad 
  |J|^{-\alpha}\,(1.110721)^s,
\]
where the implied constants are independent of $s$ under favorable conditions. Both results hold and we can interpolate between them to trade between the convergence rate and the growth in~$s$. This will be explained more precisely later in the introduction.
Here it makes sense to state the convergence rate with respect to $|J|$, which corresponds to the number of function evaluations.

\subsection{Our approach}\label{sub21}

Our approach with the full lattice $\bsX = \{\bst_0,\ldots,\bst_{N-1}\}$ can be described in three steps as follows: 
\begin{align} 
  I(f) = \int_{[0,1]^s} f(\bsx) \,\rd\bsx 
  &= \int_{[0,1]^s} f(\bspsi(\bsy))\,\mu(\bsy) \,\rd\bsy
  = \int_{[0,1]^s} F(\bsy)\,\mu(\bsy) \,\rd\bsy \label{eq:magic1} \\
  &\approx \int_{[0,1]^s} (A^\bsX_{\rm ker} F)(\bsy)\,\mu(\bsy) \,\rd\bsy 
  = \int_{[0,1]^s} \bigg(\sum_{k=0}^{N-1} a_k(F)\, K(\bst_k,\bsy)\bigg)\,\mu(\bsy) \,\rd\bsy   
  \label{eq:magic2} \\
  &= \sum_{k=0}^{N-1} a_k(F)\, \kappa(\bst_k)
  = \sum_{k=0}^{N-1} w_k\, F(\bst_k) 
  = \sum_{k=0}^{N-1} w_k\,f(\bspsi(\bst_k)) = Q^\bsX_\bspsi(f), \label{eq:magic3}
\end{align} 
where all of the ingredients $F$, $\mu$, $A^\bsX_{\rm ker}$, $K$, $a_k$, $\bst_k$, $\kappa$, $w_k$ will be defined and explained step by step below.
Later we may replace $\bsX$ by a randomly subsampled lattice $\bsX_J = \{\bst_k: k\in J\}$.

\subsubsection*{Step 1. Periodizing the integrand}

The first step is to make a change of variable,
\begin{align} \label{eq:psi}
  (x_1, \ldots, x_s)=: \bsx = \bspsi(\bsy) &\coloneqq (\psi(y_1), \ldots, \psi(y_s)), \nonumber\\
 \psi(y) &\coloneqq \sin^2(\pi y) = \frac{1}{2}(1 - \cos(2\pi y))
 \mbox{ for } y \in [0,1],
\end{align}
to convert the problem to a different integration problem over the unit cube with respect to the $\bsy$ variable as shown in \eqref{eq:magic1}. Our choice of the univariate mapping function $\psi: y \in [0,1]\to x \in [0,1]$ traverses the interval $[0,1]$ twice, see Figure~\ref{fig:trans} (left), and has the effect of making the transformed function
\begin{align} \label{eq:F}
 F(\bsy) \coloneqq  f(\bspsi(\bsy)), \quad \bsy\in [0,1]^s, 
\end{align}
a smooth $1$-periodic function in each coordinate direction. By the usual change of variable rules we then obtain a new expression for the integral in \eqref{eq:magic1}, 
with the product density 
\begin{align} \label{eq:mu}
 \mu(\bsy) \coloneqq \prod_{j = 1}^s \frac{|\psi'(y_j)|}{2} 
 = \prod_{j = 1}^s \frac{\pi\,|\sin(2\pi y_j)|}{2}.
\end{align}

\begin{figure}[t]
 \begin{center}
\begin{tikzpicture}[scale=0.8]
\begin{axis}[
    title={$\psi(y) = \sin^2(\pi y) = \frac{1}{2} (1 - \cos(2\pi y))$},
    domain=0:1,
    samples=201,
    xlabel={$y$},
    xmin=0, xmax=1,
    ymin=0, ymax=1.05,
]
\addplot[blue, thick] {sin(deg(pi*x))^2};
\end{axis}
\end{tikzpicture}
~
\begin{tikzpicture}[scale=0.8]
\begin{axis}[
    title={$\mu(y) = \frac{1}{2}|\psi'(y)| = \frac{\pi}{2} |\sin(2 \pi y)|$},
    domain=0:1,
    samples=201,
    xlabel={$y$},
    xmin=0, xmax=1,
    ymin=0, ymax=1.05,
    ytick={0, 0.25, 0.5, 0.75, 1},
    yticklabels={$0$,$\frac{\pi}{8}$,$\frac{\pi}{4}$,$\frac{3\pi}{8}$,$\frac{\pi}{2}$},
]
\addplot[blue, thick] {abs(sin(deg(2*pi*x)))};
\end{axis}
\end{tikzpicture}
 \caption{Univariate transformation $x=\psi(y)=\sin^2(\pi y) = \frac{1}{2} (1 - \cos(2\pi y))$ 
 and associated density $\mu(y) = \frac{1}{2} |\psi'(y)| = \frac{\pi}{2} |\sin(2 \pi y)|$.}  \label{fig:trans}
 \end{center}
\end{figure}
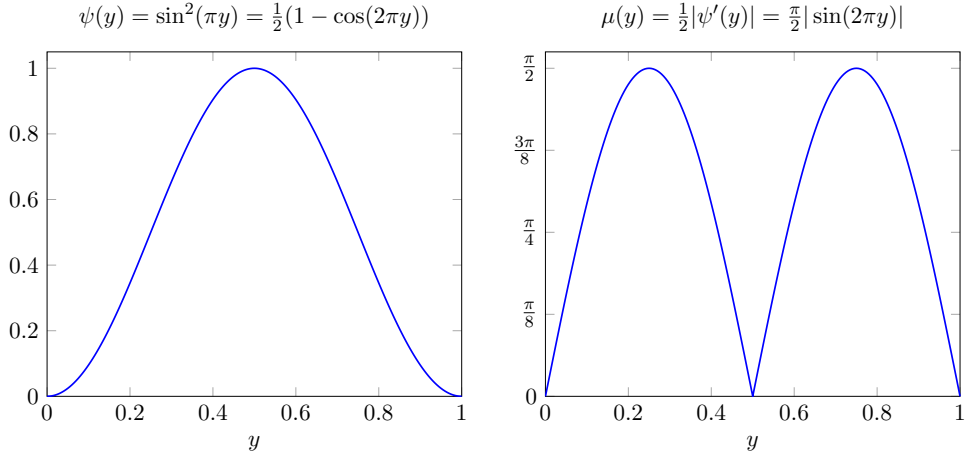

Note that the univariate density is $1/2$ times the Jacobian, see Figure~\ref{fig:trans} (right). This possibly unexpected result is a consequence of the double traverse of $\psi$ over the interval $[0,1]$. To verify this for $s=1$,
we transform the integral of $f$ by making two duplicates, adding half of each, and then applying the one-to-one mappings $\psi : [0,\frac{1}{2}] \to [0,1]$ and $\psi : [\frac{1}{2},1] \to [0,1]$ to each piece, noting  that $\psi$ increases from $0$ to $1$ on the first piece (so $\psi'$ is positive), and decreases from $1$ to $0$ on the second piece (so $\psi'$ is negative): this gives
\begin{align*}
 \int_0^1 f(x) \,\rd x 
 &= \frac{1}{2} \int_0^1 f(x) \,\rd x + \frac{1}{2} \int_0^1 f(x) \,\rd x 
 = \frac{1}{2} \int_0^{\frac{1}{2}} f(\psi(y))\, \psi'(y) \,\rd y 
   + \frac{1}{2} \int_1^{\frac{1}{2}} f(\psi(y))\, \psi'(y) \,\rd y \\
 &= \frac{1}{2} \int_0^{\frac{1}{2}} f(\psi(y))\, \psi'(y) \,\rd y 
   - \frac{1}{2} \int_{\frac{1}{2}}^1 f(\psi(y))\, \psi'(y) \,\rd y
 = \int_0^1 f(\psi(y))\, \frac{|\psi'(y)|}{2} \,\rd y.
\end{align*}

\subsubsection*{Step 2. Kernel approximation of the function $F$}

The second step is to approximate the smooth periodic function $F$ in \eqref{eq:F} by a \emph{kernel interpolant} using the values of $F$ at the points $\bsX = \{\bst_0,\ldots,\bst_{N-1}\} \subset [0,1]^s$, as shown in \eqref{eq:magic2}. This approximation takes the form
\begin{align} \label{eq:AF}
 (A^{\bsX}_{\rm ker}F)(\bsy) \coloneqq \sum_{k=0}^{N-1} a_k(F)\, K(\bst_k,\bsy), \quad \bsy\in [0,1]^s,
\end{align}
for a known kernel function $K:[0,1]^s\times [0,1]^s \to \bbR$, where the coefficients $a_k(F)\in\bbR$ are determined by the interpolation condition $(A^{\bsX}_{\rm ker}F)(\bst_{k'}) = F(\bst_{k'})$ for all $k' = 0,\ldots,N-1$. Defining the vectors $\bsa = [a_k]_{k=0}^{N-1}$, $\bsF = [F(\bst_k)]_{k = 0}^{N-1}$ and the matrix $\calK = [K(\bst_k, \bst_{k'})]_{k, k' = 0}^{N-1}$, the interpolation condition can be written as the linear system in matrix-vector form 
\begin{align}\label{eq:lin_sys}
 \calK\bsa = \bsF \implies \bsa = \calK^{-1}\bsF.
\end{align}

We assume that the transformed function $F$ belongs to the \emph{weighted Korobov space} which is a reproducing kernel Hilbert space (see Subsection~\ref{sec:kor}), and we will use its kernel $K$ in~\eqref{eq:AF}. For now it suffices to say that the kernel is \emph{shift-invariant} so that $K(\bst_k,\bst_{k'}) = K(\bst_k-\bst_{k'},\bszero)$, i.e., it depends only on the difference between the two points. For the points $\bst_k$ in \eqref{eq:AF} we will use lattice points
\begin{align} \label{eq:lat}
  \bst_k = \frac{k \bsz}{N} \bmod 1, \quad k=0,\ldots,N-1,
\end{align}
with generating vector $\bsz \in \{1, \ldots, N-1\}^s$. This pairing of a shift-invariant kernel and a lattice pointset yields a linear system with a circulant matrix $\calK$, which can be solved to obtain $\bsa = \calK^{-1}\bsF$ very efficiently using the Fast Fourier transform (FFT) in $\calO(N\log N)$ operations. We also have well established theory on the convergence of the approximation error, see Section~\ref{sec:app}.

\subsubsection*{Step 3. Exact integration of the approximate integrand}

The next step of our method is to integrate the product  of the kernel interpolant and the density $A^{\bsX}_{\rm ker} F(\bsy)\, \mu(\bsy)$ exactly to arrive at \eqref{eq:magic3}, 
with
\begin{align} \label{eq:kappa}
  \kappa(\bst) \coloneqq 
  \int_{[0,1]^s} K(\bst,\bsy)\,\mu(\bsy) \,\rd\bsy, \quad \bst\in [0,1]^s.
\end{align}
If the kernel is a product of univariate kernels (when we have product weights for the weighted Korobov space, see Subsection~\ref{sec:kor}) then this last step requires the accurate computation of only $s$ one-dimensional integrals. 

Combining $\bskappa \coloneqq [\kappa(\bst_k)]_{k=0}^{N-1}$ with $\bsa = \calK^{-1}\bsF$ in \eqref{eq:magic3}, we will show in Section~\ref{sec:kernel} that the quadrature weights $w_k$ are given collectively by $\bsw = \calK^{-1}\,\bskappa$, which depends on the kernel and the points, but conveniently is independent of $F$.
Hence, in practice we can compute the weights $w_k$ once and
then use the quadrature rule \eqref{eq:magic3} to approximate the integral, without the need to actually compute the kernel approximation \eqref{eq:AF}.
In fact, we will explain in Section~\ref{sec:kernel} that the weights given by $\bsw = \calK^{-1}\,\bskappa$ are ``optimal'' for the reproducing kernel Hilbert space corresponding to the kernel $K$,
in the sense that they minimise the worst case error for weighted integration with respect to the density $\bsmu$ in that space.

\subsubsection*{Step 2$'$ and Step 3$'$. The subsampling alternative}

Steps 1--3 explained our method for a full lattice pointset $\bsX$. However, it is well-known \cite{BKUV17} that for kernel interpolation with lattice points and Korobov kernels the worst case approximation error achieves only half the optimal convergence rate, i.e., close to $\calO(N^{-\alpha/2})$ instead of $\calO(N^{-\alpha})$. 

This issue can be addressed by using only a small subset of the lattice points. Instead of the full lattice pointset $\bsX$, we could work with a randomly subsampled lattice $\bsX_J$ for $J \subset \{0,\ldots,N-1\}$ of size $|J|\sim \sqrt{N}\log N$ (see Section~\ref{sec:app}).
To this end, we also need to replace the kernel matrix $\calK$ by the $|J|\times|J|$ sub-matrix $\calK_J$ and solve instead the linear system $$\calK_J\,\bsa_J = \bsF_J.$$ 
Again, combining $\bskappa_J \coloneqq [\kappa(\bst_k)]_{k\in J}$ with $\bsa_J = \calK_J^{-1}\bsF_J$, it can be shown that that the quadrature weights $w_{k,J}$ are given collectively by $\bsw_J = \calK_J^{-1}\,\bskappa_J$, which depends on the kernel and the subsampled points, but conveniently is again independent of $F$.
Full details will be explained in Section~\ref{sec:kernel}.
Subsampling allows us to use fewer function values, and in turn improves the overall convergence rate of our method to be close to $\calO(|J|^{-\alpha})$, see below.
 
\subsubsection*{Error analysis}

The only step to incur an error in our method \eqref{eq:magic1}--\eqref{eq:magic3} is the kernel interpolation, and thus
\begin{align} \label{eq:err-intro}
 |I(f) - Q^{\bsX}_\bspsi(f) | 
 &= \bigg| \int_{[0,1]^s} \big( F(\bsy) -  (A^{\bsX}_{\rm ker}F)(\bsy)\big) \,\mu(\bsy)\,\rd\bsy \bigg| 
 \nonumber\\ 
 &\le \inf_{p\in [2,\infty]} \big\{ \|F-A^{\bsX}_{\rm ker}F\|_{L_p([0,1]^s)}\, \|\mu\|_{L_{p'}([0,1]^s)} 
 \big\},
\end{align}
where we used H\"older's inequality with the conjugate pair $1/p + 1/p' = 1$. An analogous error bound holds if we replace the full lattice $\bsX$ by a subsampled lattice $\bsX_J$.

Both $L_2$ and $L_\infty$ approximation have been analyzed in the literature for lattice-based kernel interpolant in weighted Korobov spaces \cite{CKNS20,CKNS21,KMN25}. 
The corresponding result for $L_2$ approximation with subsampled lattice was recently proved in \cite{BGKS25}, and we extend this to $L_\infty$ approximation in this paper. We then interpolate the $L_2$ and $L_\infty$ results to obtain the $L_p$ approximation result for all $p\in [2,\infty]$. Here we need \emph{one single lattice generating vector} which is suited for all values of $p$, thus we need some adaptations of existing results.
For the same lattice we obtain upper bounds on $\|F-A^{\bsX}_{\rm ker}F\|_{L_p([0,1]^s)}$ and $\|F-A^{\bsX_J}_{\rm ker}F\|_{L_p([0,1]^s)}$ in the first factor of \eqref{eq:err-intro} for all $p$, and they depend on the norm of $F$ in the weighted Korobov space. 

We then relate this Korobov norm of the transformed $F$ to a special mixed Sobolev norm of the original integrand~$f$:
\begin{align*}
 \|F\|_{{\rm Kor},s,\alpha,\bsgamma} \le \|f\|_{{\rm Sob},s,\alpha,\bsbeta,\omega}.
\end{align*}
Details of our two function spaces will be given in Section~\ref{sec:space}. For now it suffices to know that $\bsgamma$ is a rescaling of the weights $\bsbeta$ (see~\eqref{eq:gamma_scaled} ahead), and $\omega$ is a weight function that appears in the Sobolev norm (see~\eqref{eq:sob-norm} ahead).

For the second factor in \eqref{eq:err-intro} we have (see \eqref{eq:xi-p} ahead)
\[
  \|\mu\|_{L_{p'}([0,1]^s)} = \xi_p^s, 
  \quad\mbox{with}\quad
  1 = \xi_\infty \le \xi_p \le \xi_2 =\frac{\pi}{\sqrt{8}} \le 1.110721.
\]
Thus, if $p=\infty$ then there is no dependence on~$s$; otherwise there is a mild exponential growth of at worst $(1.110721)^s$, which occurs when $p=2$.

In summary, we obtain the error bounds 
\begin{align*} 
 |I(f) - Q^{\bsX}_\bspsi(f) | 
 &\le \inf_{p\in [2,\infty]} \Big\{ C_p\, \,N^{- \theta_p}\,\, \xi_p^s \big\} \,
 \|f\|_{{\rm Sob},s,\alpha,\bsbeta,\omega},
 \quad \theta_p \to \frac{\alpha}{2} - \frac{1}{4} + \frac{1}{2p}, \\
 |I(f) - Q^{\bsX_J}_\bspsi(f) | 
 &\le \inf_{p\in [2,\infty]} \Big\{ C_p\, |J|^{- \theta_p}\, \xi_p^s \big\}\,
 \|f\|_{{\rm Sob},s,\alpha,\bsbeta,\omega},
 \quad \theta_p \to \alpha - \frac{1}{2} + \frac{1}{p},
\end{align*}
with different constants $C_p$ and rates $\theta_p$ for the two cases, and they all depend on some arbitrarily small parameter $\epsilon>0$. We have shown the limits of the rates as $\epsilon\to 0$, but the constants go to infinity as $\epsilon\to 0$.
The constants $C_p$ and the norm of $f$ can be bounded independently of $s$ under appropriate conditions on the weights $\bsgamma$ and $\bsbeta$. The only error dependence on $s$ is then through~$\xi_p^s$. The value of $p$ which yields the infimum will change depending on the size of $N$ or $|J|$ relative to~$s$. 

Moreover, if the function $f$ is \emph{twice} as smooth compared to the reproducing kernel with smoothness $\alpha$, then a generalization of the recent work \cite{VS25} yields \emph{double the rates}:
\begin{align*} 
 |I(f) - Q^{\bsX}_\bspsi(f) | 
 &\le \inf_{p\in [2,\infty]} \big\{ C_p\, \,N^{-\theta_p}\,\, \xi_p^s \,\big\} \,
 \|f\|_{{\rm Sob},s,2\alpha,\bseta^2,\omega},
 \quad \theta_p \to \alpha- \frac{1}{4} + \frac{1}{2p}, \\
 |I(f) - Q^{\bsX_J}_\bspsi(f) | 
 &\le \inf_{p\in [2,\infty]} \big\{ C_p\, |J|^{-\theta_p}\, \xi_p^s \,\big\}\,
 \|f\|_{{\rm Sob},s,2\alpha,\bseta^2,\omega},
 \quad \theta_p \to 2\alpha - \frac{1}{2} + \frac{1}{p},
\end{align*}
with new constants $C_p$, noting that the norm is now with
smoothness $2\alpha$ and squared and rescaled weights $\bseta^2$. Full details are given in Section~\ref{sec:error}. 

The different convergence rates for $p=2$ and $p=\infty$ are summarized in Table~\ref{tab:conv}, with interpolated results for the values of $p\in (2,\infty)$.
We stress that our method does not depend on the value of $p$ and we do not need to specify a value of $p$ for the error bound to hold. We will automatically get the smallest error bound over all possible values of~$p\in [2,\infty]$. There is a trade-off between the polynomial rate in $N$ or $|J|$ and the exponential dependence on $s$. 
 
\begin{table}[t]
  \centering
  \begin{tabular}{l|ll|ll}
  & \multicolumn{2}{l|}{$f$ with smoothness $\alpha$} 
  & \multicolumn{2}{l}{$f$ with smoothness $2\alpha$} \\
  & $p=2$ & $p=\infty$ & $p=2$ & $p=\infty$ \vphantom{$A^A$} \\
  \hline
  &&&& \\[-4mm]
  full lattice $\bsX$ & $N^{-\alpha/2}\,(1.110721)^s$ & 
  $N^{-(\alpha/2-1/4)}$ & $N^{-\alpha}\,(1.110721)^s$ & $N^{-(\alpha-1/4)}$ \\[1mm]
  subsampled lattice $\bsX_J$ & $|J|^{-\alpha}\,(1.110721)^s$ & 
  $|J|^{-(\alpha-1/2)}$ & $|J|^{-2\alpha}\,(1.110721)^s$ & $|J|^{-(2\alpha-1/2)}$
  \end{tabular}
  \caption{Summary of convergence results for quadrature rule using full and subsampled lattices.}
  \label{tab:conv}
\end{table}

\subsection{Structure of this paper}

In Section~\ref{sec:space} we define a weighted Korobov space and a special weighted Sobolev space, then relate the Korobov norm of $F$ to the Sobolev norm of $f$. In Section~\ref{sec:kernel} we provide further details on kernel interpolation and associated cost, as well as the exact integration to obtain our final quadrature weights $w_k$. In Section~\ref{sec:app} we provide error analysis to construct one single lattice generating vector for kernel interpolation, with a full lattice and a subsampled lattice, under both the $L_2$ and $L_\infty$ approximation settings, and with results interpolated for $L_p$ approximation for all $p\in [2,\infty]$. In Section~\ref{sec:error} we apply these approximation error bounds to obtain the final error bound for our method, given in our main result, Theorem~\ref{thm:main}, and we also obtain doubling the rates in Corollary~\ref{cor:double} if the integrand $f$ is twice as smooth compared to the kernel. In Section~\ref{sec:num} we present numerical results. In Section~\ref{sec:other} we compare our method with related strategies and give some concluding remarks. 

\section{Function spaces} \label{sec:space}

In Step~1 of our method (see \eqref{eq:magic1}) we introduced a transformation function $F = f \circ \bspsi$ in \eqref{eq:F}. We now define two Hilbert spaces needed for our analysis: a weighted Korobov space of periodic functions for our transformed function $F$, and a special weighted mixed Sobolev space for our original integrand $f$. We then relate the Korobov norm of $F$ to the Sobolev norm of $f$.

\subsection{Weighted Korobov space of periodic functions} \label{sec:kor}

Consider all periodic functions defined on $[0,1]^s$ with absolutely convergent Fourier series
\begin{align} \label{eq:Fourier}
 F(\bsy) = \sum_{\bsh\in\bbZ^s} \widehat{F}_\bsh\, e^{2\pi\ri\bsh\cdot\bsy}, 
 \qquad\text{where}\qquad
 \widehat{F}_\bsh = \int_{[0,1]^s} F(\bsy)\,e^{-2\pi\ri\bsh\cdot\bsy}\,\rd\bsy.
\end{align}
We define the \emph{weighted Korobov space} $H^{\rm Kor}_{s,\alpha,\bsgamma}$ with dimension $s\in\bbN$, smoothness parameter $\alpha>1/2$ and positive weights $\bsgamma = \{\gamma_\setu\}_{\setu\subset\bbN}$ to be the space of such functions with finite norm defined by
\begin{equation}\label{eq:kor-norm}
 \|F\|_{{\rm Kor},s,\alpha,\bsgamma}^2
 \coloneqq \sum_{\bsh\in\bbZ^s} r_{s,\alpha,\bsgamma}(\bsh) \, |\widehat{F}_{\bsh}|^2,
 \qquad
 r_{s,\alpha,\bsgamma}(\bsh)
 \coloneqq \gamma_{\supp(\bsh)}^{-1} \prod_{j\in\supp(\bsh)} |h_j|^{2\alpha},
\end{equation}
where $\supp(\bsh) \coloneqq \{ j \in\{1:s\} : h_j\ne 0\}$ and $\{1:s\} \coloneqq \{1,\ldots,s\}$. The user-defined weight $\gamma_\setu$ controls the importance of the part of the function $f$ that depends only on the variables $\bsy_\setu \coloneqq (y_j)_{j\in\setu}$. We fix $\gamma_\emptyset \coloneqq 1$.
For integer $\alpha\in\mathbb N$ it is known that
\begin{equation} \label{eq:kor-norm2}
 \|F\|_{{\rm Kor},s,\alpha,\bsgamma}^2
 = \sum_{\setu\subseteq\{1:s\}} \frac{1}{(2\pi)^{2\alpha|\setu|}} \frac{1}{\gamma_\setu}
   \int_{[0,1]^{|\setu|}}
   \bigg|
   \int_{[0,1]^{s-|\setu|}}
   \bigg(\prod_{j\in\setu} \frac{\partial^\alpha}{\partial y_j^\alpha}\bigg)
    F(\bsy)
    \,\rd\bsy_{\{1:s\}\setminus\setu}
    \bigg|^2
    \,\rd\bsy_\setu.
\end{equation}

The space $H^{\rm Kor}_{s,\alpha,\bsgamma}$ is a reproducing kernel Hilbert space, with inner product 
\[
\langle F,G\rangle_{{\rm Kor},s,\alpha,\bsgamma} \coloneqq \sum_{\bsh\in\bbZ^s}r_{s,\alpha,\bsgamma}(\bsh)\,\widehat{F}_{\bsh}\, \overline{\widehat{G}_{\bsh}}
\]
and the reproducing kernel
\begin{equation} \label{eq:kor-ker}
 K^{\rm Kor}_{s,\alpha,\bsgamma}(\bsy,\bsy')
 \coloneqq \sum_{\setu\subseteq{\{1:s\}}}\gamma_{\setu} \prod_{j\in\setu}{\rm ker}_\alpha(y_j,y_j'),
 \qquad 
 {\rm ker}_\alpha(y,y') \coloneqq \sum_{h\in\bbZ\setminus\{0\}} \frac{e^{2\pi\ri h(y-y')}}{|h|^{2\alpha}}.
\end{equation}
In general, a reproducing kernel $K(\cdot,\cdot)$ for a Hilbert space $H$ over a domain $D$ satisfies $K(\bsy,\cdot)\in H$ for all $\bsy\in D$, and $\langle K(\bsy,\cdot), F\rangle_H = F(\bsy)$ for all $F \in H$ and $\bsy\in D$; the latter is known as the reproducing property.

For integer $\alpha$ we can write
\begin{align} \label{eq:Bernoulli}
  {\rm ker}_\alpha(y,y') =
  \frac{(2\pi)^{2\alpha}}{(-1)^{\alpha+1}(2\alpha)!} B_{2\alpha}(\{y-y'\}),\qquad y,y'\in[0,1],
\end{align}
where the braces indicate that $y-y'$ is to be replaced by its fractional part in $[0,1)$, and $B_{2\alpha}(y)$ is the Bernoulli polynomial of even degree $2\alpha$. For example, for $\alpha=1,2,3$ we have
\[
 B_{2}(y) \,=\, y^{2}-y+ \frac{1}{6}, \quad
 B_{4}(y) \,=\, y^{4}-2y^{3}+y^{2}-\frac{1}{30}, \quad
 B_{6}(y) \,=\, y^6 - 3y^5 + \frac{5y^4}{2}-\frac{y^2}{2} + \frac{1}{42}.
\]

In the special case of \emph{product weights}, i.e., $\gamma_\setu = \prod_{j\in \setu}\gamma_j $ for some positive sequence $(\gamma_j)_{j\ge 1}$, the kernel takes a particularly simple form
\[
 K^{\rm Kor}_{s,\alpha,\bsgamma}(\bsy,\bsy') = \prod_{j=1}^s (1+ \gamma_j\,{\rm ker}_\alpha(y_j,y_j')).
\] 
 
\subsection{Weighted Sobolev space of non-periodic functions} \label{sec:sob}

For an integer $\alpha \geq 1$, we define a special weighted Sobolev space of $H^{\rm Sob}_{s,\alpha,\bsbeta,\omega}$ to be the class of functions on $[0,1]^s$ with mixed square-integrable derivatives up to order $\alpha$ in each coordinate, equipped with the norm 
\begin{equation} \label{eq:sob-norm}
 \|f\|_{{\rm Sob},s,\alpha,\bsbeta,\omega}^2 
 \coloneqq \sum_{\setu\subseteq\{1:s\}} \frac{1}{\beta_\setu}
 \sum_{\substack{\bsnu \in \{0:\alpha\}^s\\ \supp(\bsnu) = \setu}} 
 \int_{[0,1]^s} \Big|\frac{\partial^\bsnu}{\partial\bsx^\bsnu} f(\bsx) \Big|^2\,
 \prod_{j=1}^s \omega(x_j) \, \rd \bsx,
 \quad
 \omega(x) \coloneqq \frac{1}{\pi\sqrt{x(1-x)}}.
\end{equation}
Analogously to the weighted Korobov space, here the positive weight $\beta_\setu$ controls the importance of the part of $f$ with respect to the variables $\bsx_\setu = (x_j)_{j\in\setu}$, with $\beta_\emptyset\coloneqq1$. In addition, there is now a univariate weight function $\omega$, i.e., the Chebyshev density on $[0,1]$, which controls the behaviour of $f$ at the boundaries of the cube. 
This special space is a subspace of the Sobolev space without weight function.
The particular choice of the weight function will be seen to arise naturally in the next subsection, see Lemma~\ref{lem:norm-link}.

With weights $\beta_\setu = 1$ and weight function $\omega\equiv 1$, the norm \eqref{eq:sob-norm} is exactly the classical mixed Sobolev norm. In other words, our special norm \eqref{eq:sob-norm} is closer to the mixed Sobolev norm than the ``anchored'' norms and the ``unanchored'' or ``ANOVA'' norms typically appearing the QMC literature.

\subsection{Relating the Korobov norm of $F$ to the Sobolev norm of $f$}

In order to utilise error bounds in the weighted Korobov space, we need to relate
the Korobov norm of the transformed periodic function $F = f \circ \bspsi$ to
the Sobolev norm of the original (non-periodic) function $f$, where we recall that $\bspsi(\bsy) = (\psi(y_1),\ldots,\psi(y_s))$ is defined in \eqref{eq:psi}.

First, we show how the derivatives are related by using the Fa\`{a} di Bruno formula 
for computing the multivariate chain rule. In what follows, we write the order $\bsnu$ mixed derivative
of a function using the shorthand notation $f^{(\bsnu)}$.

\begin{lemma}
Let $f \in H^{\rm Sob}_{s,\alpha,\bsbeta,\omega}$ for integer $\alpha \geq 1$ and weights $\bsbeta$. For any multiindex $\bsnu \in \{0:\alpha\}^s$ the mixed derivative of the transformed function $F = f \circ \bspsi$ satisfies
\begin{equation} \label{eq:Fnu_bound}
 | F^{(\bsnu)}(\bsy)|
 \le (2\pi)^{|\bsnu|} \!\!
 \sum_{\substack{\bslambda \leq \bsnu\\\ \supp(\bslambda) = \supp(\bsnu)}} 
 \!\!  | f^{(\bslambda)}(\bspsi(\bsy))|\, \prod_{j=1}^s S(\nu_j,\lambda_j),
\end{equation}
where $S(\nu,\lambda)$ denotes the Stirling number of the second kind.
\end{lemma}
\begin{proof}
We follow \cite[Appendix~B]{KaaSch26}, which analyzed similar periodizing transformations
in the context of PDEs with random coefficients.
The recursive Fa\`{a} di Bruno formula from \cite{Sav06} gives
\begin{equation} \label{eq:FdB}
 F^{(\bsnu)}(\bsy) = \sum_{\substack{\bslambda \in \N_0^s \\ 1\leq |\bslambda| \leq |\bsnu}} 
 f^{(\bslambda)}(\bspsi(\bsy))\, \kappa_{\bsnu, \bslambda}(\bsy),
\end{equation}
where, since $\bspsi$ is a componentwise map, the sequence $\kappa_{\bsnu, \bslambda}(\bsy)$
can be defined recursively as $\kappa_{\bsnu,\bszero}(\bsy) \equiv \delta_{\bsnu, \bszero}$, $\kappa_{\bsnu,\bslambda}(\bsy) \equiv 0$ if $|\bsnu| < |\bslambda|$ or $\bslambda\not\ge\bszero$, and otherwise
\begin{align*}
 \kappa_{\bsnu+\bse_j,\bslambda}(\bsy) &= \sum_{m = 0}^{\nu_j} \psi^{(m + 1)}(y_j) \,
 \kappa_{\bsnu-m\bse_j,\bslambda-\bse_j}(\bsy).
\end{align*}
Here $\bse_j \in \N_0^s$ is the $j$th unit basic vector.
Note that the recursive formula also implies that $\kappa_{\bsnu, \bslambda} \equiv 0$ if 
$\supp(\bslambda) \neq \supp(\bsnu)$.

For $m \geq 1$, the derivative of $\psi$ defined in \eqref{eq:psi} can be bounded by
\[
 |\psi^{(m)}(y)| 
 = \Big|\frac{\rd^m}{\rd y^m}\big(\sin^2(\pi y)\big)\Big|
 = \Big|\frac{\rd^m}{\rd y^m}\Big(\frac{1 - \cos(2\pi y)}{2}\Big)\Big|
 = \Big|(2\pi)^m \cos\big(2\pi y + m \frac{\pi}{2}\big)\Big|
\leq (2\pi)^m,
\]
which is the same bound as satisfied by the derivatives of the transformation in \cite[Theorem~B.1]{KaaSch26}. Hence, the sequence $\kappa_{\bsnu, \bslambda}$ satisfies the same bounds
\[
 |\kappa_{\bsnu,\bslambda}(\bsy)| \le (2\pi)^{|\bsnu|} \prod_{j=1}^s S(\nu_j, \lambda_j).
\]
Substituting this into \eqref{eq:FdB} yields \eqref{eq:Fnu_bound},
where we have simplified the limits in the sum by using the property 
that $S(\nu, \lambda) = 0$ if $\nu < \lambda$ and $S(\nu,0) = \delta_{\nu,0}$.
\end{proof}

Next, we need to know under what conditions the map $\bsx = \bspsi(\bsy)$ can be inverted.
To see this in one dimension, let $g : [0, 1] \to \R$ be integrable and consider the periodized function $g \circ \psi$.
Since $g \circ \psi$ is even periodic, i.e., $g(\psi(y)) = g(\psi(1 - y))$, the integral of $g \circ \psi$ can be written as
\begin{equation}
\label{eq:int_g_split}
\int_0^1 g(\psi(y)) \, \rd y = \int_0^{1/2} g(\psi(y)) \, \rd y + \int_{1/2}^1 g(\psi(y)) \, \rd y = 2\int_0^{1/2} g(\psi(y)) \, \rd y.
\end{equation}
The mapping $\psi$ is now monotone on the smaller interval $[0, \tfrac{1}{2}]$, and so we can 
make the change of variables
\begin{align*}
 x = \psi(y) = \sin^2(\pi y) = \frac{1}{2}( 1 - \cos(2\pi y)),
 \quad
 \rd x = \pi\sin(2\pi y) \,\rd y,
\quad
y = \frac{1}{2\pi} \cos^{-1}(1 - 2x),
\end{align*}
to obtain
\begin{align*}
 \int_{0}^{1/2} g(\psi(y)) \,\rd y 
 = \int_{0}^{1} \frac{g(x)}{\pi \sin(\cos^{-1}(1 - 2x))} \, \rd x
 = \frac{1}{2\pi} \int_0^1 \frac{g(x)}{\sqrt{x(1 - x)}} \, \rd x,
\end{align*}
where we used the identity $\sin(\cos^{-1}(z)) = \sqrt{1 - z^2}$ for $z \in (-1, 1)$. Substituting this into \eqref{eq:int_g_split} leads to
\[
 \int_0^1 g(\psi(y)) \,\rd y = \int_0^1 g(x)\,\omega(x) \,\rd x,
\]
where $\omega$ is the weight function defined in \eqref{eq:sob-norm}. Thus the inverse mapping is defined provided that $g$ is integrable with respect to the weight function $\omega$.

Since $\bspsi$ is a componentwise mapping, this can be generalised to higher dimensions
to give the following lemma. Here we define the weighted $L^1$ space by
\[
 L^1_\omega([0,1]^s) 
 \coloneqq \bigg\{ g:[0,1]^s \to \R : \int_{[0, 1]^s} |g(\bsx)| \prod_{j=1}^s \omega(x_j) \,\rd \bsx < \infty \bigg\}.
\]

\begin{lemma}
\label{lem:inverse_map}
Let $g \in L^1_\omega([0,1]^s)$. Then the mapping $\bspsi$ can be ``inverted''
in the following sense
\[
 \int_{[0, 1]^s} g(\bspsi(\bsy))\,\rd\bsy 
 = \int_{[0, 1]^s} g(\bsx) \prod_{j=1}^s \omega(x_j)\,\rd\bsx,
\]
and the integral on the right is finite.
\end{lemma}

Now we can bound the Korobov norm of $F = f\circ \bspsi$ by the Sobolev norm of
$f$.

\begin{lemma} \label{lem:norm-link}
Let $f \in H^{\rm Sob}_{s,\alpha,\bsbeta,\omega}$ for integer $\alpha \geq 1$ and positive weights $\bsbeta = \{\beta_\setu\}_{\setu\subset\bbN}$. Then the Korobov norm of $F = f \circ \bspsi$ 
satisfies
\[
 \|F\|_{{\rm Kor},s,\alpha,\bsgamma} \le \|f\|_{{\rm Sob},s,\alpha,\bsbeta,\omega},
\]
where $\bsgamma = \{\gamma_\setu\}_{\setu\subset\bbN}$ are scaled weights given by
\begin{equation} \label{eq:gamma_scaled}
 \gamma_\setu \coloneqq c_\alpha^{|\setu|} \, \beta_\setu,
 \quad \text{with} \quad
 c_\alpha \coloneqq \sum_{\lambda=1}^\alpha S(\alpha,\lambda)^2.
\end{equation}
\end{lemma}

\begin{proof}
Applying the Cauchy--Schwarz inequality to the (squared) Korobov norm \eqref{eq:kor-norm2}
then substituting in the bound \eqref{eq:Fnu_bound} for each $\bsnu = \alpha \sum_{j \in \setu} \bse_j$ gives
\begin{align*}
 \|F\|_{{\rm Kor},s,\alpha,\bsgamma}^2 
 &\le
 \sum_{\setu\subseteq\{1:s\}} \frac{1}{(2\pi)^{2\alpha|\setu|}} \frac{1}{\gamma_{\setu}}
 \int_{[0,1]^s} \bigg|\bigg(\prod_{j\in\setu} \frac{\partial^\alpha}{\partial y_j^\alpha}\bigg)
 F(\bsy)\bigg|^2 \,\rd\bsy \\
 &\le \sum_{\setu\subseteq\{1:s\}} \frac{1}{\gamma_\setu}
 \int_{[0,1]^s} \bigg| \sum_{\substack{\bslambda \in \{1:\alpha\}^s\\ \supp(\bslambda) = \setu}} 
 f^{(\bslambda)}(\bspsi(\bsy)) \prod_{j \in \setu} S(\alpha, \lambda_j) \bigg|^2 \,\rd\bsy,
\end{align*}
where we have simplified the inner sum using $S(\alpha, 0) = 1$.

Applying the Cauchy--Schwarz inequality to the sum over $\bslambda$, we can bound this by
\begin{align*}
 \|F\|_{{\rm Kor},s,\alpha,\bsgamma}^2 
 &\le \sum_{\setu\subseteq\{1:s\}}
 \frac{1}{\gamma_\setu} \int_{[0,1]^s} 
 \bigg( \sum_{\substack{\bslambda \in \{1:\alpha\}^s\\ \supp(\bslambda) = \setu}}  
 |f^{(\bslambda)}(\bspsi(\bsy))|^2 \bigg)\,
  \bigg( \sum_{\substack{\bslambda \in \{1:\alpha\}^s\\ \supp(\bslambda) = \setu}}  
 \prod_{j \in \setu} S(\alpha, \lambda_j)^2 \bigg)  \,\rd\bsy \\
 &= \sum_{\setu\subseteq\{1:s\}} 
 \frac{1}{\gamma_\setu}\,
 \bigg(\sum_{\lambda=1}^\alpha S(\alpha,\lambda)^2\bigg)^{|\setu|}
 \sum_{\substack{\bslambda \in \{1:\alpha\}^s\\ \supp(\bslambda) = \setu}}
 \int_{[0,1]^s} |f^{(\bslambda)}(\bspsi(\bsy)) |^2 \,\rd\bsy \\
 &= \sum_{\setu\subseteq\{1:s\}} \frac{1}{\beta_\setu}
    \sum_{\substack{\bslambda \in \{1:\alpha\}^s\\ \supp(\lambda) = \setu}}\int_{[0,1]^s}
    |f^{(\bslambda)}(\bsx) |^2 \prod_{j=1}^s \omega(x_j)\, \rd\bsx
= \| f\|_{{\rm Sob},s,\alpha,\bsbeta,\omega}^2,
\end{align*}
where we used the connection between weights given in \eqref{eq:gamma_scaled}, as well as Lemma~\ref{lem:inverse_map},
which is valid since $f \in H^{\rm Sob}_{s,\alpha,\bsbeta,\omega}$ implies that
$|f^{(\bslambda)}(\cdot)|^2 \in L_\omega^1([0, 1]^s)$ for all $\bslambda \in \{0:\alpha\}^s$.
\end{proof}

Note that the norm inequality in Lemma~\ref{lem:norm-link} is sharp for the constant function.

\section{Kernel interpolation and integration} \label{sec:kernel}

Here we provide further details for Steps~2 and~3 of our method.

\subsection{Kernel interpolation}

In Step~2 of our method (see \eqref{eq:magic2}) we approximated the transformed function $F$ by its kernel interpolant
\begin{equation*}
    (A^{\bsX}_{\rm ker}F)(\bsy) \coloneqq  \sum_{k = 0}^{N-1} a_k(F)\, K(\bst_k, \bsy).
\end{equation*}
Defining $\bsa = [a_k(F)]_{k = 0}^{N-1}$, $\bsF = [F(\bst_k)]_{k = 0}^{N-1}$, and $\calK = [K(\bst_k,\bst_{k'})]_{k,k'=0}^{N-1}$, we already explained that
the coefficients $a_k(F)$ are given by solving the linear system $\calK \bsa =\bsF$ to obtain $\bsa = \calK^{-1}\bsF$. 

In the case of a Korobov kernel and a full lattice $\bsX$,
we have $K(\bsx,\bsy) = K(\bsx-\bsy,\bszero)$ and thus $\calK$ is a circulant matrix
\begin{equation*}
    \calK
    = \big[K(\bst_k, \bst_{k'})\big]_{k,k'=0}^{N-1}  
    = \big[K(\bst_{(k-k')\bmod N}, \bszero)\big]_{k,k'=0}^{N-1},
\end{equation*}
i.e., each entry only depends on the difference of the row and column indices.
Thus, it is possible to diagonalize it by a discrete Fourier transform, which can be performed in almost linear cost using the fast Fourier transform (FFT).
The kernel approximation with lattices is therefore computable using three FFTs in $\mathcal O(N\log N)$ arithmetic operations,
\[
\bsa = \calK^{-1} \bsF = \mathtt{ifft} (\mathtt{fft}(\bsF) ./ \mathtt{fft}(\bsc)),
\]
where $\bsc = [ K(\bst_k,\boldsymbol 0) ]_{k=0}^{N-1}$ is the first column of $\calK$, and $./$ denotes componentwise division, and where $\mathtt{fft}$ and $\mathtt{ifft}$ denote the FFT and inverse FFT, respectively.

The cost for evaluating the kernel will depend on the structure of the weights $\gamma_{\mathfrak u}$, ranging from linear in $s$ for product weights to quadratic in $s$ for POD and SPOD weights, see \cite[Table~1]{KKKNS22}.

Kernel methods have a long history, tracing back to Wahba's seminal work on splines~\cite{GW}. They have an important role in radial basis functions \cite{HW}, and have recently been used in high dimensional approximation, see, e.g., \cite{ZLH06, ZKH09,GRZ19, KKKNS22, KKS23}. We will provide error analysis for lattice-based kernel methods in Section~\ref{sec:app}.

\subsection{The numerical integration step} \label{sec:Cheb-lattice}

In Step 3 of our method (see \eqref{eq:magic3}) we integrated the product of the kernel interpolant 
$A^\bsX_{\rm ker}F$ and the density $\bsmu$ exactly. Thus our initial integral $I(f)$ has been approximated by 
\begin{align} \label{eq:quad1}
 I_\bsmu(F) \coloneqq \int_{[0,1]^s} F(\bsy)\,\mu( \bsy) \,\rd \bsy 
 \approx \sum_{k=0}^{N-1} a_k(F)\,\kappa(\bst_k), \quad
 \kappa(\bst_k)
 = \int_{[0,1]^s} K(\bst_k, \bsy)\,\mu(\bsy) \,\rd\bsy.
\end{align}
In our case, since the Korobov kernel $K = K_{{\rm Kor},s,\alpha,\bsgamma}$ can be written explicitly in terms of Bernoulli polynomials (see \eqref{eq:Bernoulli}), the integrals $\kappa(\bst_k)$ can be evaluated analytically at essentially no cost. 

Since the interpolant coefficients are linear in the function evaluations, by defining $\bskappa = [\kappa(\bst_k)]_{k = 0}^{N-1}$, we can write 
\begin{equation*}
 \sum_{k = 0}^{N-1} a_k(F)\, \kappa(\bst_k) 
 = \bsa^\top \bskappa \,=\, \bsF^\top \calK^{-1} \bskappa
 = \sum_{k = 0}^{N-1} [\calK^{-1}\bskappa]_k \, F(\bst_k).
\end{equation*}
Hence the approximation \eqref{eq:quad1} is equivalent to a weighted quadrature rule
\begin{equation} \label{eq:quad2}
Q^\bsX_\bspsi(f) \coloneqq \sum_{k = 0}^{N-1} w_k\, f(\bspsi(\bst_k)) 
= \sum_{k = 0}^{N-1} w_k\, F(\bst_k),  
\quad\mbox{with}\quad
  w_k \coloneqq [\calK^{-1}\bskappa]_k.
\end{equation}

Since the points $\bst_k$ are lattice points (see \eqref{eq:lat}), the transformed points $\psi(\bst_k)$ of the weighted quadrature rule are related to a \emph{Chebyshev lattice} \cite{CooPop2011}, see Section~\ref{sec:cheb}.  The effect of the transformation is to concentrate the points near the boundary and since $\psi$ traverses the interval twice there is a duplication of points.  A $2$-dimensional example  of a lattice and the corresponding transformed lattice points are shown in Figure~\ref{fig:2d_lattice}.

\begin{figure}[t]
 \centering
 \begin{tikzpicture}[scale=5]
        \draw (0, 0) -- (1, 0);
        \draw (0, 0) -- (0, 1);
        \draw (0, 1) -- (1, 1);
        \draw(1, 0) -- (1, 1);
    \foreach \k in {0,..., 143}{
    	\pgfmathsetmacro{\x}{mod(1*\k/144, 1)}
    	\pgfmathsetmacro{\y}{mod(89*\k/144, 1)}    	
    		\fill[black] (\x, \y) circle (.5pt);
    }
\end{tikzpicture}
~~~~~~~~
 \begin{tikzpicture}[scale=5]
        \draw (0, 0) -- (1, 0);
        \draw (0, 0) -- (0, 1);
        \draw (0, 1) -- (1, 1);
        \draw(1, 0) -- (1, 1);
    \foreach \k in {0,..., 143}{
    	\pgfmathsetmacro{\x}{0.5*(1 - cos(deg(2*pi*mod(1*\k/144, 1))))}
    	\pgfmathsetmacro{\y}{0.5*(1 - cos(deg(2*pi*mod(89*\k/144, 1))))}    	
    		\fill[black] (\x, \y) circle (.5pt);
    }
\end{tikzpicture}
\caption{Two-dimensional lattice rule with $N = 144$ and $\bsz = (1, 89)^\top$ on the left, and the same lattice points after applying the transformation $\bspsi$ on the right, showing $\lfloor 144/2+1\rfloor = 73$ distinct points.}
\label{fig:2d_lattice}
\end{figure}
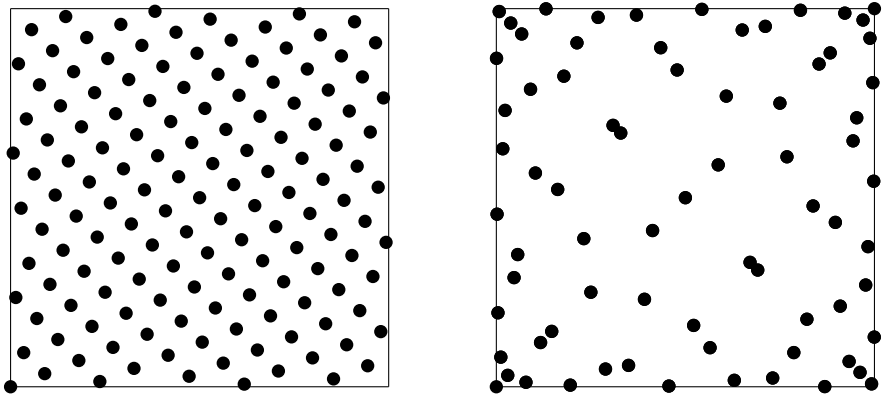
The weights $\bsw = [w_k]_{k = 0}^{N - 1}$ given by \eqref{eq:quad2} are in fact \emph{optimal weights} for $\bsmu$-weighted integration in the Korobov space $H^{\rm Kor}_{s,\alpha,\bsgamma}$ with the pointset $\bsX \coloneqq \{\bst_0, \ldots, \bst_{N-1}\}$
as studied in \cite{GriebSeid24,Oett-diss}. 
By ``optimal'' we refer here to the weights that minimise the \emph{worst case integration error} of the quadrature rule for weighted integration with respect to $\bsmu$ in the space $H^{\rm Kor}_{s,\alpha,\bsgamma}$,
\begin{equation*} 
 e^{{\rm int}, \bsmu}(Q^\bsX_\bsmu, H^{\rm Kor}_{s,\alpha,\bsgamma}) 
\coloneqq
  \sup_{F \in H^{\rm Kor}_{s,\alpha,\bsgamma},\, \|F\|_{{\rm Kor},s,\alpha,\bsgamma} \le 1} 
    \bigg| I_\bsmu(F) - \sum_{k = 0}^{N - 1} w_k\, F(\bst_k) \bigg|.
\end{equation*}

\subsection{Implementation for full lattice}
\label{sec:imp-full}

To construct the quadrature rule \eqref{eq:quad2} for the full lattice $\bsX =\{\bst_0,\ldots,\bst_{N-1}\}$, we must choose the generating vector $\bsz$ to obtain the lattice points $\{\bst_k\}$ by \eqref{eq:lat} and then compute the weights $\{w_k\}$.
The error analysis and construction for $\bsz$ will be given in Section~\ref{sec:app} below. For now it suffices to say that $\bsz$ can be constructed efficiently using the fast \emph{component-by-component} (CBC) algorithm for approximation, which for product, POD and SPOD weights is outlined in \cite{CKNS20}. The cost of the CBC construction depends on the weights and is summarised in Table~\ref{tab:cost-full}.

\begin{table}[t]
\centering 
\begin{tabular}{l||c|c|c}
Step $\backslash$ Weights $\bsgamma$ & product & POD & SPOD
\\
\hline
CBC & $s N \log(N)$ & $sN \log(N) + s^2 \log(s) N$ & $sN\log(N) + s^3 \alpha^2 N$
\\
Compute $K(\bst_k, \bszero)$, $\kappa(\bst_k)$ for all $k$ & $s N $ & $sN + s^2 \log(s) N$ & $s^2 \alpha^2 N$
\\
Compute weights $\bsw$ & $N \log(N)$ & $N \log(N)$ & $N \log(N)$
\\
\hline
\textbf{Total} & $s N \log(N)$ & $sN \log(N) + s^2 \log(s) N$ & $sN\log(N) + s^3 \alpha^2 N$
\end{tabular}
\caption{Cost of steps to construct quadrature rule for full lattice \eqref{eq:quad2} for different weights $\bsgamma$.}
\label{tab:cost-full}
\end{table}

Then, given lattice points, to compute the corresponding weights $\{w_k\}$ we first compute the kernel integrals $\kappa(\bst_k)$ from \eqref{eq:quad1}.
Expanding the kernel as in \eqref{eq:kor-ker} and using the fact that $\int_0^1 \mu(y) \, \rd y = 1$, we have
\begin{equation}
\label{eq:kap-sum}
\kappa(\bst_k) = \sum_{\setu \subseteq \{1:s\}} \gamma_\setu  \prod_{j \in \setu} \int_0^1 {\rm ker}_\alpha(t_{k, j}, y_j)\, \mu(y_j) \, \rd y_j 
=  \sum_{\setu \subseteq \{1:s\}} \gamma_\setu  \prod_{j \in \setu} \Lambda_\alpha(t_{k, j}),
\end{equation}
where for $t \in [0,1)$ we define the one-dimensional integral
\begin{equation}
\label{eq:Lambda}
\Lambda_\alpha(t) \coloneqq \int_0^1 {\rm ker}_\alpha(t, y)\, \mu(y) \, \rd y
= \frac{\pi (2\pi)^{2\alpha}}{2(-1)^{\alpha + 1} (2\alpha) !}
\int_0^1 B_{2\alpha}( \{t - y\}) \,|\sin(2\pi y)| \, \rd y,
\end{equation}
which can be computed exactly. In particular, 
\begin{align*}
&\int_0^1 B_{2\alpha}( \{t - y\}) \,|\sin(2\pi y)| \, \rd y \\
&= \begin{cases}
\int_0^t B_{2\alpha}(t-y) \,\sin(2\pi y) \, \rd y
 + \int_t^{1/2} B_{2\alpha}(t-y+1) \,\sin(2\pi y) \, \rd y
 - \int_{1/2}^1 B_{2\alpha}(t-y+1) \,\sin(2\pi y) \, \rd y \\
\hfill \mbox{if } t\in [0,1/2], \\
 \int_0^{1/2} B_{2\alpha}(t-y) \,\sin(2\pi y) \, \rd y
 - \int_{1/2}^t B_{2\alpha}(t-y) \,\sin(2\pi y) \, \rd y
 - \int_t^1 B_{2\alpha}(t-y+1) \,\sin(2\pi y) \, \rd y \\
\hfill \mbox{if } t \in [1/2,1).
 \end{cases}
\end{align*}
For $\alpha = 1, 2, 3$ and $t \in [0, 1)$ these integrals are given by
\begin{align*}
 \Lambda_1(t) 
 &= \frac{\pi^2}{3} + \frac{\pi}{2} |\sin(2\pi t)|
    +  \frac{\pi^2}{2} \Big(4t^2 - 4t - 2\,\Big|t - \frac{1}{2}\Big| + \frac{1}{2} - \frac{2}{\pi^2}\Big),
\\
 \Lambda_2(t)
&= \frac{-\pi^4}{3} \Big[ 2t^4 - 4t^3 + \Big(\frac{7}{2} - \frac{6}{\pi^2}\Big)t^2 
        - \Big(\frac{3}{2} - \frac{6}{\pi^2}\Big)t
        - 2\,\Big|t - \frac{1}{2}\Big|^3 + \frac{3}{\pi^2} \Big|t - \frac{1}{2}\Big|
        \nonumber\\
 &\qquad\qquad+ \frac{59}{240} - \frac{7}{4\pi^2} + \frac{3}{\pi^4} - \frac{3}{2\pi^3} |\sin(2\pi t)|\Big],
\end{align*}
and, denoting the indicator function by $\mathbbm{1}(\cdot)$,
\begin{align*}
\Lambda_3(t) = &\frac{2}{45} \bigg( 2 \pi^6 t^6 - 6 \pi^6 t^5 + \bigg(\frac{65}{4} \pi^6 - 15 \pi^4\bigg) t^4 
    + \bigg(\frac{-45}{2} \pi^6 + 30 \pi^4\bigg) t^3 
    + \bigg(\frac{299}{16} \pi^6 - \frac{195}{4} \pi^4 + 45 \pi^2\bigg) t^2
    \\&
    + \bigg(\frac{-135}{16} \pi^6 + \frac{135}{4} \pi^4 - 45 \pi^2\bigg) t 
    + \frac{2143}{1344} \pi^6 - \frac{299}{32} \pi^4 + \frac{195}{8} \pi^2 - \frac{45}{2}
    \\
    &-3 \pi  \bigg(\mathbbm{1}(t \leq \tfrac{1}{2}) \bigg[\frac{-15}{4}  \sin(2 \pi  t) 
    - \pi^5 t^5 
    + 5 \pi^5 t^4 
    + \bigg( 5 \pi^3 - \frac{15}{2} \pi^5\bigg) t^3 
    + \bigg(-15 \pi^3 + \frac{25}{4} \pi^5\bigg) t^2 
    \\&
    + \bigg(\frac{-15}{2} \pi  + \frac{45}{4} \pi^3 - \frac{45}{16} \pi^5\bigg) t 
    + \frac{15}{2} \pi  - \frac{25}{8} \pi^3 + \frac{17}{32} \pi^5\bigg]
    \\&
    + \mathbbm{1}(t > \tfrac{1}{2}) \bigg[ \frac{15}{4}  \sin(2 \pi  t)
    +  \pi^5 t^5
    + \bigg( - 5 \pi^3 -  \frac{5}{2} \pi^5\bigg) t^3 
    + \frac{15}{4} \pi^5 t^2 
    \\&
    + \bigg(\frac{15}{2} \pi  + \frac{15}{4} \pi^3 - \frac{35}{16} \pi^5\bigg) t - \frac{15}{8} \pi^3 + \frac{15}{32} \pi^5 \bigg)\bigg]\bigg)\bigg).
\end{align*}

Thus the integrals $\kappa(\bst_k)$ can be computed via \eqref{eq:kap-sum}, where depending on the type of weights $\bsgamma$ the sum over $\setu$ can be computed efficiently in the same way as evaluating the kernel (see Table \ref{tab:cost-full} for the cost).
Note that all components of the lattice points $\bst_k$ are multiples of $1/N$, so we only need to compute $\Lambda_\alpha(t)$ for $t\in \{0, \frac{1}{N}, \frac{2}{N}, \ldots, \frac{N - 1}{N}\}$.

Next we compute the weights $\bsw = [w_k]_{k = 0}^{N - 1}$ by solving the system $\calK \bsw = \bskappa$, which since $\calK$ is circulant can be done in (almost) linear time using the FFT. Let $\bsc = [K(\bst_k, \bszero)]_{k = 0}^{N - 1}$ denote the first column of the kernel matrix as above, then we can also solve the linear system for the quadrature weights using the FFT, 
\begin{equation}
\label{eq:w_fft}
\bsw = \texttt{ifft}\big( \texttt{fft}(\bskappa) ./ \texttt{fft}(\bsc)\big),
\end{equation}
with cost $\calO(N \log N)$.

The procedure to compute the points and weights for the quadrature rule \eqref{eq:quad2} is summarised in Algorithm~\ref{alg:full}. Since the linear system can be solved in $\calO(N \log N)$ the cost of constructing the quadrature rule is dominated by the CBC construction cost. The cost breakdown is summarised in Table~\ref{tab:cost-full}.

\begin{algorithm}[t]
\caption{Constructing the quadrature points and weights for the full (or subsampled) lattice}
\begin{algorithmic}[1]
\Statex \textbf{Input:} $\alpha$, $\bsgamma$, $N$ (and $J$ for the subsampled variant)
\Statex \textbf{Output:} $\bsz$, $\bsX = \{\bst_k\}_{k = 0}^{N - 1}$, $\bsw = [w_k]_{k = 0}^{N - 1}$
\State construct $\bsz$ using (fast) CBC algorithm
\State \textbf{for} $t \in\{0, \frac{1}{N}, \frac{2}{N}, \ldots, \frac{N - 1}{N}\}$ \textbf{do} compute $\Lambda_\alpha(t)$ in \eqref{eq:Lambda}
\For{$k = 0, 1, 2, \ldots, N - 1$}
\State compute $\bst_k = \displaystyle \frac{k\bsz}{N} \bmod 1$
\State compute $K(\bst_k, \bszero)$ using \eqref{eq:kor-ker}
\State compute $\kappa(\bst_k)$ using \eqref{eq:kap-sum}
\EndFor
\State for the full lattice compute the weights $\bsw$ via fast FFT using \eqref{eq:w_fft}
(for the subsample variant compute the weights $\bsw_J$ in \eqref{eq:w_J} via one of three methods)
\end{algorithmic}
\label{alg:full}
\end{algorithm}

\subsection{Implementation for subsampled lattice}
\label{sec:imp-sub}
For a subset $J \subset \{0, 1, \ldots, N - 1\}$, the quadrature rule for the subsampled lattice $\bsX_J = \{\bst_k\}_{k \in J}$ is
\begin{equation} \label{eq:quad_sub}
Q^{\bsX_J}_\bspsi(f) \coloneqq \sum_{k \in J} w_{J, k}\, f(\bspsi(\bst_k)), \quad\mbox{with}\quad
  w_{J, k} \coloneqq [\calK_J^{-1}\bskappa_J]_k,
\end{equation}
where now the weights $\bsw_J = [w_{J, k}]_{k \in J}$ come from the subsampled kernel approximation and we denote $\calK_J = [K(\bst_k, \bst_{k'})]_{k, k' \in J}$ and $\bskappa_J = [\kappa(\bst_k)]_{k \in J}$. Here we have followed the same steps as in \eqref{eq:quad2} to obtain the subsampled weights. 

Algorithm~\ref{alg:full} Steps 1--7 remain the same for the subsampled lattice, however, Step~8 is now replaced by solving
\begin{equation}
\label{eq:w_J}
\calK_J\, \bsw_J = \bskappa_J,
\end{equation}
using one of three methods:
\begin{enumerate}
\item Solve \eqref{eq:w_J} directly with cost $\calO(|J|^3)$.
\item Solve \eqref{eq:w_J} using an iterative method, such as conjugate gradient, with a cost $\calO(|J|^2)$ times the number of iterations. Note that for an iterative method the number of iterations will depend on the condition number of the matrix $\calK_J$, which can be large in practice. 
\item Solve \eqref{eq:w_J} using an iterative method but to perform the matrix vector multiplications at each iteration by embedding $\calK_J$ back into the full kernel matrix $\calK$ and then exploiting the FFT. For $\bsv \in \R^{|J|}$, we can perform the matrix-vector multiplication $\calK_J\,\bsv$ by embedding $\calK_J$ into $\calK$ and padding $\bsv$ with zeroes, performing the multiplication with the full matrix $\calK$ (using the FFT) then selecting the entries corresponding to the set $J$. 
Mathematically, we compute
\[
 \calK_J \, \bsv = [\calK \, (\bsv; \bszero_{-J})]_J = [ \texttt{ifft}( \texttt{fft}(\bsc) .\!* \texttt{fft}((\bsv; \bszero_{-J}))]_J
\]
where $.\!*$ denotes componentwise multiplication and $(\bsv; \bszero_{-J})_k = v_k$ if $k \in J$ and 0 otherwise. Performing the matrix-vector multiplication in this way costs $\calO(N \log N)$ due to the FFT and so the overall cost of solving \eqref{eq:w_J} is the same order as computing the weights for the full lattice via \eqref{eq:w_fft}. Note again, however, that this cost will still depend on the number of iterations, which can be high in practice.
\end{enumerate}

To compare the cost of the three methods in terms of both the size of the subsampled lattice $|J|$ and the size of the full lattice $N$, we use the property that the subsampling theory in Section~\ref{sec:app} below implies the upper bound $|J| \lesssim N^{\frac{1 + \delta}{2}} \log N$ for some $\delta > 0$. Specifically, from Theorem~\ref{thm:combined} below it follows that $\delta = \epsilon/(\alpha - \epsilon)$ for $\epsilon > 0$ from the convergence rate. Using this upper bound on $|J|$ the cost of each method for solving \eqref{eq:w_J} is summarised in Table~\ref{tab:subsys_cost}. Note that in practice, similar to \cite{BGKS25} we take $|J| = \lceil \sqrt{N} \log N \rceil$, i.e., we treat $\delta$ as 0.

\begin{table}[t]
\centering
\begin{tabular}{l||c|c}
Method & Cost in $|J|$ & Cost in $N$ 
\\
\hline
1. direct & $|J|^3$ & $\displaystyle N^{\frac{3}{2}(1 + \delta)} \log(N)^3$ 
\\
2. iterative  &  \# iter.~$\times \, |J|^2 $ & \# iter.~$ \times \, \displaystyle N^{1 + \delta} (\log N)^2$ 
\\
3. iterative  with  FFT & N/A & \# iter.~$\times \, N \log N$ 
\end{tabular}
\caption{Summary of cost of different methods for solving the subsampled system \eqref{eq:w_J} in terms of size of the subsampled lattice $|J|$ and full lattice $N$. In practice we take $|J| = \sqrt{N}\log N$ so $\delta = 0$.}
\label{tab:subsys_cost}
\end{table}

Thus, depending on the method used to solve \eqref{eq:w_J} the cost of constructing the quadrature rule will be the same as the full lattice outlined in Table~\ref{tab:cost-full} with Method~3 
or possibly higher with Methods 1~and~2. 
The benefit of using the subsampled rule is that, after the quadrature rule is constructed, only $|J| = \sqrt{N}\log N$ function evaluations are required which is much less.

\section{Theory for lattice-based kernel interpolant and subsampling} \label{sec:app}

In Step~2 of our method (see \eqref{eq:magic2}) we approximated the transformed function $F$ by its kernel interpolant $A^\bsX_{\rm ker}F$, while the alternative Step~2$'$ replaced $\bsX$ by a subsampled lattice $\bsX_J$. In this section we provide the full error analysis for both.

To obtain error bounds for a kernel interpolant with a full lattice, we make use of error bounds from a ``classical lattice algorithm''. To obtain error bounds for a kernel interpolant with a subsampled lattice, we need error bounds from a ``least squares approximation''. We will discuss these in the next two subsections. We will need to adapt the existing theory to have one single lattice generating vector for all values of $p\in [2,\infty]$ in \eqref{eq:err-intro}.

We will start with a general function space setting, and later specialize to the weighted Korobov space.
For functions defined on $[0,1]^s$ with absolutely convergent Fourier series \eqref{eq:Fourier},
we consider a general function space $H$ equipped with the norm
\begin{align} \label{eq:H-norm}
  \|F\|_H^2
  \coloneqq \sum_{\bsh\in\bbZ^s} r(\bsh)\,|\widehat{F}_\bsh|^2,
\end{align}
where the ``radius'' function $r(\bsh)$ enforces a decay in the Fourier coefficients, with $r(\bsh)\to\infty$ for $\|\bsh\|_2 \to \infty$. The specific choice of $r(\bsh) = r_{s,\alpha,\bsgamma}(\bsh)$ defined in \eqref{eq:kor-norm} yields the weighted Korobov space.

\subsection{Classical lattice algorithm and kernel interpolant}

For a given frequency index set $\calA\subset\bbZ^s$ and a lattice $\bsX = \{\frac{k}{N}\bsz\bmod 1\}_{k=0}^{N-1}$, a classical lattice algorithm approximates $F$ by
\begin{align} \label{eq:classical}
  (A_{\calA}^{\bsX} F)(\bsy) \coloneqq \sum_{\bsh\in\calA}
    \bigg( \frac{1}{N} \sum_{k=0}^{N-1} F(\bst_k)\, e^{-2\pi\ri k\bsh\cdot\bsz/n} \bigg)\,   
    e^{2\pi\ri\bsh\cdot\bsy},
    \quad \bst_k = \frac{k}{N}\bsz \bmod 1,
\end{align}
see, e.g., \cite{KSW06, KaPoVo13, kaemmererdiss, PlPoStTa18, CKNS20, DKP22}. For the same lattice $\bsX = \{\frac{k}{N}\bsz\bmod 1\}_{k=0}^{N-1}$, recall that a kernel interpolant $A_{\rm ker}^\bsX F$ approximates $F$ by \eqref{eq:AF}, where the coefficients $a_k(F)$ are obtained by solving a linear system $(A_{\rm ker}^\bsX F)(\bst_{k'}) = F(\bst_{k'})$ for $k'=0,\ldots,N-1$.

The \emph{worst case $L_p$ approximation error} of any algorithm $A$ in the space $H$ is defined as
\[
  e^{\rm app}(A,H,L_p) \coloneqq \sup_{\|F\|_H \le 1} \| F - AF \|_{L_p}.
\]
Recall that the kernel interpolant is optimal among all algorithms using the same function evaluations, see e.g., \cite[Theorem~2.2]{KKKNS22}. We have the following approximation error bounds.

\begin{theorem} \label{thm:full}
For a frequency index set $\calA \coloneqq \{\bsh\in\bbZ^s : r(\bsh)\le M\}$ with $M>0$ and a lattice $\bsX = \{\frac{k}{N}\bsz\bmod 1\}_{k=0}^{N-1}$, we have
\begin{align}
  e^{\rm app}(A_{\rm ker}^\bsX,H,L_2) \le e^{\rm app}(A_\calA^\bsX,H,L_2)
  &< \bigg(\frac{1}{M} + M\, \calS_N(\bsz) \bigg)^{1/2}, \label{eq:A-L2}
  \\
  e^{\rm app}(A_{\rm ker}^\bsX,H,L_\infty) \le e^{\rm app}(A_\calA^\bsX,H,L_\infty)
  &< \bigg(\sum_{\bsh\notin\calA} \frac{1}{r(\bsh)}\bigg)^{1/2} 
  + \sqrt{|\calA|}\,\sqrt{M\,\calS_N(\bsz)},   \label{eq:A-Linf}
\end{align}  
where
\begin{align}\label{eq:S}
  \calS_N(\bsz)
  \coloneqq \sum_{\bsh\in\bbZ^s} \frac{1}{r(\bsh)} 
  \sum_{\substack{\bsell\in\bbZ^s\setminus\{\bszero\}\\\bsell\cdot\bsz\equiv 0 \!\!\!\!\pmod N}} 
  \frac{1}{r(\bsh+\bsell)}.
\end{align}
    In particular, the choice $M = 1/\sqrt{\calS_N(\bsz)}$ minimizes the right-hand side of \eqref{eq:A-L2} and yields
\begin{align} \label{eq:AS-L2}
  e^{\rm app}(A_{\rm ker}^\bsX,H,L_2) \le e^{\rm app}(A_\calA^\bsX,H,L_2)
  < \sqrt{\frac{2}{M}} = \sqrt{2}\,(\calS_N(\bsz))^{1/4}.
\end{align}
\end{theorem}

\begin{proof}
For $F\in H$ we write $F-A_\calA^\bsX F = (F-P_\calA F) + (P_\calA F-A_\calA^{\bsX}F)$, where $(P_\calA F)(\bsy) = \sum_{\bsh\in\calA} \widehat{F}_\bsh e^{2\pi\ri\bsh\cdot\bsy}$ is the $L_2$-projection of $F$ onto ${\rm span}\{e^{2\pi\ri\bsh\cdot\bsy}: \bsh\in\calA\}$. The $L_2$ bound \eqref{eq:A-L2} was derived in \cite[Section~2.4]{CKNS20}. We have $\|F-A_\calA^\bsX F\|_{L_2}^2 = \|F - P_\calA F\|_{L_2}^2 + \|P_\calA F - \calA_\calA^\bsX\|_{L_2}^2$, with
\begin{align} \label{eq:L2-step}
  \|F - P_\calA F\|_{L_2}^2 
  = \sum_{\bsh\not\in\calA} \frac{r(\bsh)}{r(\bsh)} \,|\widehat{F}_\bsh|^2 \le \frac{1}{M}\,\|F\|_H^2
 \quad\mbox{and}\quad
  \|P_\calA F - \calA_\calA^\bsX\|_{L_2}^2 < M\,\calS_N(\bsz)\,\|F\|_H^2.
\end{align}
We state the latter as a strict inequality, because in the derivation of the quantity $\calS_N(\bsz)$, see e.g., \cite[Formula~(2.7)]{CKNS20}, infinitely many positive terms with index $\bsh\notin\calA$ were added. 

We now proceed to derive the $L_\infty$ bound \eqref{eq:A-Linf}. For $F\in H$, by the triangle inequality we have
\begin{align} \label{eq:tri-A}
  \|F-A_\calA^{\bsX}F\|_{L_\infty} \le \|F-P_\calA F\|_{L_\infty} 
  + \|P_\calA F-A_\calA^{\bsX}F\|_{L_\infty}.
\end{align}
For the first term on the right-hand side of \eqref{eq:tri-A}, we use the Cauchy--Schwarz inequality to obtain
\begin{align*}
  \|F-P_\calA F\|_{L_\infty} 
  \le \sum_{\bsh\not\in\calA} \sqrt{\frac{r(\bsh)}{r(\bsh)}}\,|\widehat{F}_\bsh| 
  \le \bigg(\sum_{\bsh\not\in\calA} \frac{1}{r(\bsh)}\bigg)^{1/2}\,\|F\|_H.
\end{align*} 
For the second term on the right-hand side of \eqref{eq:tri-A}, we write $(P_\calA F - A_\calA^{\bsX}F)\,(\bsy) =: \sum_{\bsh\in\calA} \widehat{q}_\bsh\,e^{2\pi\ri\bsh\cdot\bsy}$ and apply the Cauchy--Schwarz inequality again to obtain
\begin{align*}
 \|P_\calA F - A_\calA^{\bsX}F\|_{L_\infty}
 \le \sum_{\bsh\in\calA} |\widehat{q}_\bsh|
 \le \sqrt{|\calA|}\,\bigg(\sum_{\bsh\in\calA} |\widehat{q}_\bsh|^2\bigg)^{1/2}
 = \sqrt{|\calA|}\; \|P_\calA F - A_\calA^{\bsX}F\|_{L_2},
\end{align*}
which can be bounded using \eqref{eq:L2-step}. 
Combining the two estimates gives \eqref{eq:A-Linf} as required.
\end{proof}
  
We remark that \cite[Lemma~5.1 and Lemma~5.4]{KMN25} included two alternative upper bounds for~\eqref{eq:A-Linf}. The first bound there is similar to \eqref{eq:A-Linf} (there is an extra factor $3$ but the square root is outside the sum which is tighter). The second bound there makes use of the precise definition of $r(\bsh)$ in weighted Korobov spaces and requires a reparametrization of the quantity $\calS_N(\bsz)$ with respect to the weights $\bsgamma$ and the smoothness parameter $\alpha$, and it yields a slightly better overall convergence rate when $\alpha>1$. We do not consider this approach in this paper, because later we need both the $L_2$ and $L_\infty$ bounds to hold for the \emph{same} lattice generating vector $\bsz$ obtained from the criterion $\calS_N(\bsz)$, rather than two different ones obtained from different criteria.

A lattice $\bsX = \{\frac{k}{N}\bsz\bmod 1\}_{k=0}^{N-1}$ is said to have the \emph{reconstructing property} on a frequency index set $\calB\subseteq\bbZ^s$ if and only if 
\begin{align}\label{eq:recon}
 \bsh\cdot\bsz \not\equiv \bsh'\cdot\bsz \!\!\!\!\pmod N
 \quad\mbox{for all}\quad
 \bsh,\bsh'\in \calB \quad\mbox{with}\quad \bsh\ne\bsh'.
\end{align}
In other words, every index has a unique value of dot product with $\bsz$ modulo $N$. We say that there are ``no aliasing frequencies'' in the index set. When this holds, the classical lattice algorithm $A^\bsX_\calB$ is a projection onto ${\rm span}\{e^{2\pi\ri\bsh\cdot\bsy} : \bsh\in\calB\}$, see e.g., \cite[Theorem~3.2]{BGKS25}. In the following, we will write $\calB$ for an index set when we have the reconstruction property and $\calA$ for a generic index set. 

It is known from \cite[Theorem~3.3]{BGKS25} that any lattice $\bsX$ has the reconstructing property on $\{\bsh\in\bbZ^s : r(\bsh) < e^{\rm app}(A^\bsX, H, L_2)^{-2}\}$, where $A^\bsX$ can be any linear or nonlinear algorithm using~$\bsX$.
In particular, \cite[Corollary~3.4]{BGKS25} applied this result to the algorithm $A^\bsX_\calA$ satisfying \eqref{eq:AS-L2}, implying that the lattice $\bsX = \{\frac{k}{N}\bsz\bmod 1\}_{k=0}^{N-1}$ has the reconstructing property on 
\begin{align} \label{eq:BS}
  \calB \coloneqq \bigg\{\bsh\in\bbZ^s : r(\bsh) \le \frac{1}{2\sqrt{\calS_N(\bsz)}} \bigg\},
\end{align}
where $\calS_N(\bsz)$ is defined by \eqref{eq:S}. This index set $\calB$ has half the radius of the index set $\calA$ satisfying~\eqref{eq:AS-L2}. From \eqref{eq:A-L2} with this half radius $M = 1/(2\sqrt{\calS_N(\bsz)})$, we obtain for the index set \eqref{eq:BS},
\begin{align*}
  e^{\rm app}(A_{\rm ker}^\bsX,H,L_2) \le e^{\rm app}(A_\calB^\bsX,H,L_2)
  < \sqrt{\frac{1.25}{M}} = \sqrt{2.5}\,(\calS_N(\bsz))^{1/4},
\end{align*}  
which differs from \eqref{eq:AS-L2} only by a constant scaling factor, as expected.

\subsection{Subsampled least squares approximation and subsampled kernel interpolant}

Next we use a subset of the lattice points $\bsX_J = \{\frac{k}{N}\bsz\bmod 1\}_{k\in J}$ for $J\subseteq\{0,\ldots,N-1\}$ and consider a \emph{least squares approximation} on a frequency index set $\calB\subseteq\bbZ^s$,
\begin{align}\label{eq:lsq}
 (S_\calB^{\bsX_J} F)(\bsy)
 \coloneqq g(\bsy) = \sum_{\bsh\in\calB} \widehat{g}_\bsh\,e^{2\pi\ri\bsh\cdot\bsy},
\end{align}
where the coefficients $[\widehat{g}_\bsh]_{\bsh\in\calB}$ are obtained by minimizing $\sum_{k\in J} | g(\bst_k) - F(\bst_k) |^2$. 

\begin{theorem} \label{thm:sub1}
Let $\calB\subseteq\bbZ^s$ be a nonempty frequency index set (not necessarily defined by \eqref{eq:BS}) for which the lattice $\bsX = \{\frac{k}{N}\bsz\bmod 1\}_{k=0}^{N-1}$ has the reconstructing property. For $J\subseteq\{0,\ldots,N-1\}$ a multiset with $|J|\ge |\calB|$, define the matrices
\begin{align} \label{eq:L-Phi}
 \bsL_{J,\calB} \coloneqq \big[ e^{2\pi\ri k \bsh\cdot\bsz/n} \big]_{k\in J,\,\bsh\in\calB}
 \qquad\mbox{and}\qquad
 \bsPhi_{J,\calB} \coloneqq \bigg[ \frac{e^{2\pi\ri k \bsh\cdot\bsz/n}}{\sqrt{r(\bsh)}} \bigg]_{k\in J,\,\bsh\in\calB}.
\end{align}
Suppose $\bsL_{J,\calB}$ has full rank $|\calB|$. Then the least squares approximation $S_\calB^{\bsX_J} F$ with the subsampled lattice $\bsX_J = \{\frac{k}{N}\bsz\bmod 1\}_{k\in J}$ is a projection onto ${\rm span}\{e^{2\pi\ri\bsh\cdot\bsy} : \bsh\in\calB\}$, with $[\widehat{g}_\bsh]_{\bsh\in\calB} = (\bsL_{J,\calB}^*\,\bsL_{J,\calB})^{-1}\bsL_{J,\calB}^*\, [ F(\bst_k) ]_{k\in J}$. Moreover, we have
\begin{align}
  e^{\rm app}(A_{\rm ker}^{\bsX_J},H,L_2) \le e^{\rm app}(S_\calB^{\bsX_J},H,L_2)
  &< \bigg(\sup_{\bsh\notin\calB} \frac{1}{r(\bsh)}
   + \frac{\|\bsPhi_{J,\calB}\|_2^2}{\sigma^2_{\min}(\bsL_{J,\calB})} \bigg)^{1/2}, \label{eq:B-L2}
  \\
  e^{\rm app}(A_{\rm ker}^{\bsX_J},H,L_\infty) \le e^{\rm app}(S_\calB^{\bsX_J},H,L_\infty)
  &< \bigg(\sum_{\bsh\notin\calB} \frac{1}{r(\bsh)}\bigg)^{1/2} + 
  \sqrt{|\calB|}\,\frac{\|\bsPhi_{J,\calB}\|_2}{\sigma_{\min}(\bsL_{J,\calB})},
  \label{eq:B-Linf}
\end{align}  
where $\sigma_{\min}(\bsL_{J,\calB})$ denotes the smallest singular value of $\bsL_{J,\calB}$.
\end{theorem}

\begin{proof}
The $L_2$ bound \eqref{eq:B-L2} was derived in \cite[Lemma~4.1, Lemma~4.2, and Theorem~4.3]{BGKS25}. The $L_\infty$ bound \eqref{eq:B-Linf} can be obtained analogously to the proof of Theorem~\ref{thm:full}.
%
\end{proof}

\begin{theorem} \label{thm:sub2}
For a lattice $\bsX = \{\frac{k}{N}\bsz\bmod 1\}_{k=0}^{N-1}$, define the frequency index set $\calB\subseteq\bbZ^s$ by \eqref{eq:BS} (thus $\bsX$ has the reconstructing property on $\calB$). For $t\ge 4$, let $J\subseteq\{0,\ldots,N-1\}$ be a multiset of uniformly i.i.d.\ drawn integers with
\[
  |J| \coloneqq \lceil 12\, |\calB|\, (\log |\calB| + t) \rceil.
\] 
Then for the subsampled lattice $\bsX_J = \{\frac{k}{N}\bsz\bmod 1\}_{k\in J}$, with probability exceeding $1 - 3\exp(-t)$ we have
\begin{align}
  e^{\rm app}(A_{\rm ker}^{\bsX_J},H,L_2) \le e^{\rm app}(S_\calB^{\bsX_J},H,L_2)
  &\le \bigg(\frac{7}{|\calB|} \sum_{\bsh\notin\calB} \frac{1}{r(\bsh)}
   + 14\,\sqrt{\calS_N(\bsz)}\bigg)^{1/2}, \label{eq:BS-L2}
  \\
  e^{\rm app}(A_{\rm ker}^{\bsX_J},H,L_\infty) \le e^{\rm app}(S_\calB^{\bsX_J},H,L_\infty)
  &\le (1+\sqrt{7})\,\bigg(\sum_{\bsh\notin\calB} \frac{1}{r(\bsh)}\bigg)^{1/2}
   + \sqrt{12}\,\sqrt{|\calB|}\,(\calS_N(\bsz))^{1/4}.
  \label{eq:BS-Linf}
\end{align}  
\end{theorem}

\begin{proof}
The $L_2$ bound \eqref{eq:BS-L2} was derived in \cite[Theorem~4.7]{BGKS25} using \cite[Lemma~4.4, Theorem~4.5, and Lemma~4.6]{BGKS25}. We now derive the $L_\infty$ bound \eqref{eq:BS-Linf}. From the proof of \cite[Theorem~4.7]{BGKS25}, we conclude with probability exceeding $1-3\exp(-t)$ that
\begin{align*}
 \frac{\|\bsPhi_{J,\calB}\|_2}{\sigma_{\min}(\bsL_{J,\calB})}
 &\le \bigg(\frac{7}{|\mathcal B|} \sum_{\bsh\notin\calB} \frac{1}{r(\bsh)}
       + 4 \Big( \sup_{\bsh\notin\calB} \frac{1}{r(\bsh)} + \sqrt{\calS_N(\bsz)} \Big)\bigg)^{1/2} \\
 &\le \bigg(\frac{7}{|\mathcal B|} \sum_{\bsh\notin\calB} \frac{1}{r(\bsh)}
        + 12 \sqrt{\calS_N(\bsz)}\bigg)^{1/2} 
 \le \frac{\sqrt{7}}{\sqrt{|\mathcal B|}} \bigg(\sum_{\bsh\notin\calB} \frac{1}{r(\bsh)}\bigg)^{1/2}
      + \sqrt{12}\, (\calS_N(\bsz))^{1/4}.
\end{align*}
Substituting this bound into \eqref{eq:B-Linf} yields \eqref{eq:BS-Linf} as required.
\end{proof}

\begin{remark}
    The random techniques used in Theorem~\ref{thm:sub2} enforce a logarithmic oversampling, i.e., the number of points $|J|$ is larger than the number of frequencies $|\mathcal B|$ by a logarithmic factor.
    This logarithmic gap was decreased in \cite{NSU21} in a non-constructive way and later in \cite{BSU23} in a constructive way.
    Shortly after it was shown that the logarithmic gap can be eliminated altogether in \cite{DKU23}, again in a non-constructive way.
    Recently, this was also made constructive in \cite{CDKU26}.
    Note that all mentioned constructive approaches rely on the so-called BSS algorithm, which has arithmetic complexity of $\mathcal O(|\mathcal B|^3 N)$.
    For this reason we favor the random approach and concede the loss in error decay.
\end{remark}

\subsection{Specializing to weighted Korobov space}

We now make use of the specific definition of $r(\bsh) = r_{s,\alpha,\bsgamma}(\bsh)$ in \eqref{eq:kor-norm} for the weighted Korobov space and apply known results from \cite{KMN25} for all the required quantities. Note that the smoothness parameter $\alpha$ in \cite{KMN25} corresponds to our $2\alpha$ here. 

Given $N\ge 2$, $s\in\bbN$, $\alpha>1/2$ and weights $\bsgamma = \{\gamma_\setu\}_{\setu\subseteq\bbN}$, the generating vector $\bsz\in\bbZ^s$ obtained from the CBC construction following \cite[Algorithm~3.2]{KMN25} satisfies \cite[Theorem~3.3]{KMN25}:
\begin{align} \label{eq:S-bound}
 \frac{2\zeta(2\alpha)\gamma_{\{1\}}}{N^{2\alpha}} 
 \le \calS_N(\bsz)
 \le \Big( \frac{\Phi_{1,\lambda} }{\varphi(N)} \Big)^{\frac{1}{\lambda}} 
 \quad\mbox{for all } \lambda\in (\tfrac{1}{2\alpha},1], 
\end{align}
where $\varphi(N) = |\{z\in \{1,\ldots, N\} : \gcd(z,N)=1\}|$ is the Euler totient function (e.g., $\varphi(N) = N-1$ for $N$ prime), $\zeta(x) = \sum_{h=1}^{\infty}h^{-x}$ is the Riemann zeta function, and
\begin{align}\label{eq:C}
 \Phi_{1,\lambda}
 \coloneqq (2^{4\alpha\lambda+1}+1)
 \bigg(\sum_{\emptyset\ne\setu\subseteq\{1:s\}} |\setu|\,\gamma_\setu^\lambda\,
 (2\zeta(2\alpha\lambda))^{|\setu|}\bigg)
 \bigg(\sum_{\setu\subseteq\{1:s\}} \gamma_\setu^\lambda\, (2\zeta(2\alpha\lambda))^{|\setu|}\bigg).
\end{align}

For an index set defined by
$\calA \coloneqq \{\bsh\in\bbZ^s : r_{s,\alpha,\bsgamma}(\bsh) \le M_\calA\}$ with $M_\calA \ge 1$,
it follows from \cite[Lemma~5.2]{KMN25} and \cite[Lemma~5.3]{KMN25} that for all $\lambda\in (\frac{1}{2\alpha},1)$
we have
\begin{align} 
 (\gamma_{\{1\}} M_\calA)^{\frac{1}{2\alpha}}
 \le |\calA|
 \le (\Phi_{2,\lambda}\, M_\calA)^{\lambda}, \qquad
 &\Phi_{2,\lambda}
 \coloneqq \bigg(\sum_{\setu\subseteq\{1:s\}} \gamma_{\setu}^{\lambda} \, (2\zeta(2\alpha\lambda))^{|\setu|}
 \bigg)^{\frac{1}{\lambda}},
 \label{eq:A-bound}
 \\
 \sum_{\bsh\notin\calA} 
 \frac{1}{r_{s,\alpha,\bsgamma}(\bsh)}
 \le \Big(\frac{\Phi_{3,\lambda} }{M_\calA}\Big)^{\frac{1-\lambda}{2\alpha\lambda}}, \qquad
 &\Phi_{3,\lambda}
 \coloneqq \frac{1}{\gamma_{\{1\}}} \bigg( (\tfrac{\lambda}{1-\lambda})^\lambda
 \sum_{\setu\subseteq\{1:s\}} \gamma_\setu^\lambda \, (2\zeta(2\alpha\lambda))^{|\setu|} 
 \bigg)^{\frac{2\alpha}{1-\lambda}}. \label{eq:tail}
\end{align}

We combine all results in a unified theorem for one single lattice generating vector $\bsz\in\bbZ^s$. For the full lattice in $L_2$ and $L_\infty$ approximations we recover the convergence rates from \cite{KMN25}. For the subsampled lattice in $L_2$ approximation we recover the convergence rate from \cite{BGKS25}. The constants here may differ from \cite{KMN25,BGKS25} due to our need to ensure that all bounds hold for the same $\bsz$. The results for the subsampled lattice in $L_\infty$ approximation, and all results for $L_p$ approximation for $p\in (2,\infty)$ are new. 

In comparison, it is known \cite{KPUU} that the optimal polynomial rate for any algorithm with $N$ function evaluations is $N^{-\alpha}$ for $L_2$ approximation and $N^{-(\alpha-1/2)}$ for $L_\infty$ approximation.

\begin{theorem} \label{thm:combined}
Let $H^{\rm Kor}_{s,\alpha,\bsgamma}$ be the weighted Korobov space defined in Section~\ref{sec:kor} with $\alpha>1/2$ and positive weights $\bsgamma = (\gamma_\setu)_{\setu\subset\bbN}$.
For $N\ge 7$ sufficiently large, obtain $\bsz\in\bbZ^s$ by the component-by-component construction in \cite[Algorithm~3.2]{KMN25} so that
$\calS_N(\bsz) = \calS_{s,\alpha,\bsgamma,N}(\bsz)\le 1/4$ satisfies \eqref{eq:S-bound}. Define
\begin{align} \label{eq:def-AB}
 \calA \coloneqq \bigg\{ \bsh\in\bbZ^s : r_{s,\alpha,\bsgamma}(\bsh) \le \frac{1}{\sqrt{\calS_N(\bsz)}} \bigg\}
 \quad\mbox{and}\quad
 \calB \coloneqq 
 \bigg\{ \bsh\in\bbZ^s : r_{s,\alpha,\bsgamma}(\bsh) \le \frac{1}{2\sqrt{\calS_N(\bsz)}} \bigg\}.
\end{align}
Then, for all $p\in [2,\infty]$ and arbitrarily small $\epsilon>0$, we have the following results:
\begin{itemize}
\item [\textnormal{(a)}]
The full lattice $\bsX = \{\frac{k}{N}\bsz\bmod 1\}_{k=0}^{N-1}$ satisfies
\begin{align} \label{eq:A-Lp}
  e^{\rm app}(A_{\rm ker}^\bsX,H^{\rm Kor}_{s,\alpha,\bsgamma},L_p) 
  &\le e^{\rm app}(A_\calA^\bsX,H^{\rm Kor}_{s,\alpha,\bsgamma},L_p)
  \le
  C^{\rm full}_{p,\epsilon}\,\Big(\frac{1}{\varphi(N)}\Big)^{\theta^{\rm full}_{p,\epsilon}},
  \nonumber\\
  \theta^{\rm full}_{p,\epsilon}
  &= \begin{cases}
  \frac{1}{2}(\alpha-\epsilon) &\mbox{if } p=2, \\
  \frac{1}{2}(\alpha-\frac{1}{2} - \epsilon)(1- \frac{\epsilon}{\alpha}) &\mbox{if } p=\infty, \\
  \tfrac{1}{2}(\alpha-\tfrac{1}{2}+\tfrac{1}{p}-\epsilon(1-\tfrac{2}{p}))(1-\tfrac{\epsilon}{\alpha})  &\mbox{if } p\in [2,\infty].
  \end{cases}
\end{align}

\item [\textnormal{(b)}]
For $t\ge 4$, let $J\subseteq\{0, \dots, N-1\}$ be a multiset of uniformly i.i.d.\ drawn integers with
\begin{align*}
  |J| \coloneqq \lceil 12\,|\calB|\,(\log|\calB| + t) \rceil.
\end{align*}
The subsampled lattice $\bsX_J = \{\frac{k}{N}\bsz\bmod 1\}_{k\in J}$ satisfies, with probability exceeding $1-3\exp(-t)$,
\begin{align} \label{eq:B-Lp}
  e^{\rm app}(A_{\rm ker}^{\bsX_J},H^{\rm Kor}_{s,\alpha,\bsgamma},L_p) 
 &\le e^{\rm app}(S_\calB^{\bsX_J},H^{\rm Kor}_{s,\alpha,\bsgamma},L_p)
 \le C^{\rm sub}_{p,\epsilon}\, \Big(\frac{\log|J|+t}{|J|}\Big)^{\theta^{\rm sub}_{p,\epsilon}},
 \nonumber\\
  \theta^{\rm sub}_{p,\epsilon}
  &= \begin{cases}
  (\alpha-\epsilon)(1-\frac{\epsilon}{\alpha}) &\mbox{if } p=2, \\
  (\alpha-\frac{1}{2}-\epsilon)(1-\frac{\epsilon}{\alpha}) &\mbox{if } p=\infty, \\
  (\alpha-\frac{1}{2}+ \frac{1}{p}-\epsilon)(1-\frac{\epsilon}{\alpha})  &\mbox{if } p\in [2,\infty].
  \end{cases}
\end{align}  
If $N$ is a prime power, then we have 
\begin{align*}
 N^{\frac{1-\epsilon/\alpha}{2}} \log N
 \lesssim |J|
 \lesssim N^{\frac{1}{2(1-\epsilon/\alpha)}} \log N 
 \quad\mbox{and}\quad
 \Big(\frac{|J|}{\log|J|}\Big)^{2(1-\epsilon/\alpha)}
 \lesssim N
 \lesssim \Big(\frac{|J|}{\log|J|}\Big)^{\frac{2}{1-\epsilon/\alpha}} \,.
\end{align*}
\end{itemize}
All the constants $C^{\rm full}_{p,\epsilon}$ and $C^{\rm sub}_{p,\epsilon}$ are bounded independently of $s$ if the weights $\bsgamma$ satisfy
\begin{align} \label{eq:tract}
 \sum_{\setu\subset\bbN,\, |\setu|<\infty} 
 \gamma_\setu^{\frac{1}{2(\alpha-\epsilon)}} 
 \big[2\,e^{1/e}\,\zeta\big(\tfrac{2\alpha}{2(\alpha-\epsilon)}\big)\big]^{|\setu|} 
 < \infty,
\end{align}
and the constants go to infinity as $\epsilon\to 0$.
\end{theorem}

\begin{proof}
For both the full lattice and the subsampled lattice, we consider the cases $p=2$ and $p=\infty$ separately and then use interpolation to obtain the result for general $p\in [2,\infty]$. In each case below, for the parameter $\lambda$ in \eqref{eq:S-bound}--\eqref{eq:tail} we substitute $\lambda = \frac{1}{2(\alpha-\epsilon)}$ with arbitrarily small $\epsilon>0$.

\begin{itemize}
\item \textsc{Full lattice with $p=2$.} We substitute the upper bound in \eqref{eq:S-bound} into \eqref{eq:AS-L2} to obtain \eqref{eq:A-Lp}, with
\begin{align*}
 \sqrt{2}\,\Big(\frac{\Phi_{1,\lambda}}{\varphi(N)}\Big)^{\frac{1}{4\lambda}} 
 =: C^{\rm full}_{2,\epsilon}\Big(\frac{1}{\varphi(N)}\Big)^{\theta^{\rm full}_{2,\epsilon}},
 \qquad
 \theta^{\rm full}_{2,\epsilon} \coloneqq \tfrac{1}{4\lambda} = \tfrac{1}{2}(\alpha-\epsilon).
\end{align*}

\item \textsc{Full lattice with $p=\infty$.} We use $M_\calA = 1/\sqrt{\calS_N(\bsz)} \ge 2$, substitute \eqref{eq:tail} and the upper bound in \eqref{eq:A-bound} into \eqref{eq:A-Linf}, and apply the upper bound in \eqref{eq:S-bound} to obtain \eqref{eq:A-Lp}, with
\begin{align*}
 \Big(\frac{\Phi_{3,\lambda}}{M_\calA}\Big)^{\frac{1-\lambda}{4\alpha\lambda}} 
 + (\Phi_{2,\lambda} M_\calA)^{\frac{\lambda}{2}} 
 \Big(\frac{M_\calA}{M_\calA^2}\Big)^{\frac{1}{2}} 
 &\le \big(\Phi_{3,\lambda}^{\frac{1-\lambda}{4\alpha\lambda}} 
 + \Phi_{2,\lambda}^{\frac{\lambda}{2}} \big)
 \Big(\frac{1}{M_\calA}\Big)^{\frac{1-\lambda}{4\alpha\lambda}} \\
 &\le \big(\Phi_{3,\lambda}^{\frac{1-\lambda}{4\alpha\lambda}} 
 + \Phi_{2,\lambda}^{\frac{\lambda}{2}} \big)
 \Big( \frac{\Phi_{1,\lambda}}{\varphi(N)}\Big)^{\frac{1}{2\lambda}\frac{1-\lambda}{4\alpha\lambda}} 
 =: C^{\rm full}_{\infty,\epsilon}\Big(\frac{1}{\varphi(N)}\Big)^{\theta^{\rm full}_{\infty,\epsilon}},
\end{align*}
where $\theta^{\rm full}_{\infty,\epsilon} \coloneqq \frac{1-\lambda}{8\alpha\lambda^2} = \frac{1}{2}(\alpha-\frac{1}{2} - \epsilon)(1- \frac{\epsilon}{\alpha})$.

\item \textsc{Full lattice with $p\in [2,\infty]$.} We use the interpolation bound
$\|g\|_{L_p} \le \|g\|_{L_2}^{\frac{2}{p}}\,\|g\|_{L_\infty}^{1-\frac{2}{p}}$ to conclude that
$C^{\rm full}_{p,\epsilon} 
\coloneqq (C^{\rm full}_{2,\epsilon})^{\frac{2}{p}}\, (C^{\rm full}_{\infty,\epsilon})^{1-\frac{2}{p}}$, and
\begin{align*}
 \theta^{\rm full}_{p,\epsilon} 
 \coloneqq \theta^{\rm full}_{2,\epsilon} (\tfrac{2}{p}) + \theta^{\rm full}_{\infty,\epsilon}(1-\tfrac{2}{p})
  &= \tfrac{1}{2}(\alpha-\epsilon) \tfrac{2}{p} 
  + \tfrac{1}{2}(\alpha-\tfrac{1}{2} - \epsilon)(1- \tfrac{\epsilon}{\alpha})(1-\tfrac{2}{p}) \\
 &= \tfrac{1}{2}(\alpha-\tfrac{1}{2}+\tfrac{1}{p}-\epsilon(1-\tfrac{2}{p}))(1-\tfrac{\epsilon}{\alpha}). 
\end{align*}

\item \textsc{Subsampled lattice with $p=2$.} For the set~$\calB$ we have from \eqref{eq:A-bound} and \cite[Theorem~4.5]{BGKS25} that
\begin{align}  \label{eq:JB}
  \frac{\gamma_{\{1\}}}{|\calB|^{2\alpha}} \le \frac{1}{M_\calB} 
  \le \Phi_{2,\lambda}\,\Big( \frac{1}{|\calB|}\Big)^{\frac{1}{\lambda}}
  \quad\mbox{and}\quad
  \frac{6\,(\log|J|+t)}{|J|} \le \frac{1}{|\calB|} \le \frac{13\,(\log|J|+t)}{|J|}.
\end{align}
We use $M_\calB = 1/(2\sqrt{\calS_N(\bsz)})\ge 1$, substitute \eqref{eq:tail} and the lower bound in \eqref{eq:A-bound} into \eqref{eq:BS-L2}, and apply the upper bounds in \eqref{eq:JB} to obtain \eqref{eq:B-Lp}, with
\begin{align*}
  &\bigg( 
  \frac{7}{(\gamma_{\{1\}} M_\calB)^{\frac{1}{2\alpha}}}
  \Big(\frac{\Phi_{3,\lambda}}{M_\calB}\Big)^{\frac{1-\lambda}{2\alpha\lambda}}
  + \frac{7}{M_\calB} \bigg)^{\frac{1}{2}}
  \le \bigg( 
  \frac{7\,\Phi_{3,\lambda}^{\frac{1-\lambda}{2\alpha\lambda}}}{\gamma_{\{1\}}^{\frac{1}{2\alpha}}}
  + 7 \bigg)^{\frac{1}{2}}
  \Big( \frac{1}{M_\calB}\Big)^{\frac{1}{4\alpha\lambda}} \\
  &\le \bigg( 
  \frac{7\,\Phi_{3,\lambda}^{\frac{1-\lambda}{2\alpha\lambda}}}{\gamma_{\{1\}}^{\frac{1}{2\alpha}}}
  + 7 \bigg)^{\frac{1}{2}}
  \Phi_{2,\lambda}^{\frac{1}{4\alpha\lambda}}\,13^{\frac{1}{4\alpha\lambda^2}}
  \Big( \frac{\log|J|+t}{|J|}\Big)^{\frac{1}{4\alpha\lambda^2}} 
  =: C^{\rm sub}_{2,\epsilon}\Big(\frac{\log|J|+t}{|J|}\Big)^{\theta^{\rm sub}_{2,\epsilon}},
\end{align*}
where $\theta^{\rm sub}_{2,\epsilon} \coloneqq \frac{1}{4\alpha\lambda^2} = (\alpha-\epsilon)(1-\frac{\epsilon}{\alpha})$.

\item \textsc{Subsampled lattice with $p=\infty$.} We use $M_\calB = 1/(2\sqrt{\calS_N(\bsz)})\ge 1$, substitute \eqref{eq:tail} into \eqref{eq:BS-Linf}, and apply again the upper bounds in \eqref{eq:JB} to obtain \eqref{eq:B-Lp}, with
\begin{align*}
  & (1+\sqrt{7})
  \Big(\frac{\Phi_{3,\lambda}}{M_\calB}\Big)^{\frac{1-\lambda}{4\alpha\lambda}}
  + \sqrt{12}\sqrt{|\calB|}\, \Big(\frac{1}{2M_\calB}\Big)^{\frac{1}{2}} 
  \le \Big( (1+\sqrt{7})\, \Phi_{3,\lambda}^{\frac{1-\lambda}{4\alpha\lambda}}\,
  \Phi_{2,\lambda}^{\frac{1-\lambda}{4\alpha\lambda^2}}
  + \sqrt{6}\, \Phi_{2,\lambda}^{\frac{1}{2\lambda}} \Big) 
  \Big(\frac{1}{|\calB|}\Big)^{\frac{1-\lambda}{4\alpha\lambda^2}} \\
  &\le \Big( (1+\sqrt{7})\, \Phi_{3,\lambda}^{\frac{1-\lambda}{4\alpha\lambda}}\,
  \Phi_{2,\lambda}^{\frac{1-\lambda}{4\alpha\lambda^2}}
  + \sqrt{6}\, \Phi_{2,\lambda}^{\frac{1}{2\lambda}} \Big) 
  13^{\frac{1-\lambda}{4\alpha\lambda^2}} 
  \Big(\frac{\log|J|+t}{|J|}\Big)^{\frac{1-\lambda}{4\alpha\lambda^2}} 
  =: C^{\rm sub}_{\infty,\epsilon}\Big(\frac{\log|J|+t}{|J|}\Big)^{\theta^{\rm sub}_{\infty,\epsilon}},
\end{align*}
where $\theta^{\rm sub}_{\infty,\epsilon} \coloneqq \frac{1-\lambda}{4\alpha\lambda^2}
= (\alpha-\frac{1}{2}-\epsilon)(1-\frac{\epsilon}{\alpha})$.

\item \textsc{Subsampled lattice with $p\in [2,\infty]$.} Again we use interpolation to conclude that
$C^{\rm sub}_{p,\epsilon} 
\coloneqq (C^{\rm sub}_{2,\epsilon})^{\frac{2}{p}}\, (C^{\rm sub}_{\infty,\epsilon})^{1-\frac{2}{p}}$, and
\begin{align*}
 \theta^{\rm sub}_{p,\epsilon} 
 \coloneqq \theta^{\rm sub}_{2,\epsilon} (\tfrac{2}{p}) + \theta^{\rm sub}_{\infty,\epsilon}(1-\tfrac{2}{p})
  &= (\alpha-\epsilon)(1- \tfrac{\epsilon}{\alpha}) \tfrac{2}{p} 
  + (\alpha-\tfrac{1}{2} - \epsilon)(1- \tfrac{\epsilon}{\alpha})(1-\tfrac{2}{p}) \\
 &= (\alpha-\tfrac{1}{2}+\tfrac{1}{p}-\varepsilon)(1-\tfrac{\epsilon}{\alpha}). 
\end{align*}
\end{itemize}

From the definition of $J$ and the upper and lower bounds in \eqref{eq:A-bound} and \eqref{eq:S-bound}, we have the following connection between $|J|$ and $N$:
\begin{align*}
 12\,|\calB|\,(\log|\calB|+t) &\le |J| \le 13\,|\calB|\,(\log|\calB| + t), 
 \\
 \Big(\frac{\varphi(N)}{\Phi_{1,\lambda}}\Big)^{\frac{1}{4\alpha\lambda}} 
 \Big(\frac{\gamma_{\{1\}}}{2}\Big)^{\frac{1}{2\alpha}}
 \le \Big(\frac{\gamma_{\{1\}}}{2\sqrt{\calS_N(\bsz)}}\Big)^{\frac{1}{2\alpha}}
 &\le |\calB|
 \le \Big(\frac{\Phi_{2,\lambda}}{2\sqrt{\calS_N(\bsz)}}\Big)^\lambda 
 \le \frac{\Phi_{2,\lambda}^\lambda\,N^{\alpha\lambda}}{(8\zeta(2\alpha)\gamma_{\{1\}})^{\frac{\lambda}{2}}}.
\end{align*}
This yields
\begin{align*}
 (\varphi(N))^{\frac{1}{4\alpha\lambda}}\,\log\varphi(N)
 \lesssim |J| \lesssim N^{\alpha\lambda} \log N ,
 \quad
 \Big(\frac{|J|}{\log|J|}\Big)^{\frac{1}{\alpha\lambda}}
 \lesssim N \quad\mbox{and}\quad 
 \varphi(N)
 \lesssim \Big(\frac{|J|}{\log|J|}\Big)^{4\alpha\lambda}.
\end{align*}
If $N = q^m$ is a prime power, then $\varphi(N) = q^m(1-\frac{1}{q}) \ge \frac{N}{2}$, so
$\frac{1}{N}\le \frac{1}{\varphi(N)} \le \frac{2}{N}$, which simplifies the connection between $|J|$ and $N$.

The constants $\Phi_{1,\lambda}$, $\Phi_{2,\lambda}$, $\Phi_{3,\lambda}$ all depend on $s$, $\alpha$, $\bsgamma$ and $\lambda= \frac{1}{2(\alpha-\epsilon)}$ with arbitrarily small $\epsilon>0$.
Consequently, the constants $C^{\rm full}_{2,\epsilon}$, $C^{\rm full}_{\infty,\epsilon}$, $C^{\rm sub}_{2,\epsilon}$, $C^{\rm sub}_{\infty,\epsilon}$ all depend on $s$, $\alpha$, $\bsgamma$ and $\epsilon$. For $p\in (2,\infty)$, the constants $C^{\rm full}_{p,\epsilon}$, $C^{\rm sub}_{p,\epsilon}$ all depend on $s$, $\alpha$, $\bsgamma$, $\epsilon$, and depend trivially on $p$. To ensure that all constants are bounded independently of $s$, we use $\max(1,|\setu|)\le (e^{1/e})^{|\setu|}$ and then assume that the weights $\bsgamma$ satisfy \eqref{eq:tract}. This completes the proof.
\end{proof}

The estimate $(\log|J| + t) \le c_\delta\,|J|^\delta$ holds for $c_\delta>0$ and arbitrarily small $\delta>0$ and $|J|$ sufficiently large. This can be used to simplify \eqref{eq:B-Lp} into 
\begin{align} \label{eq:B-Lp-delta}
  e^{\rm app}(A_{\rm ker}^{\bsX_J},H^{\rm Kor}_{s,\alpha,\bsgamma},L_p) 
 &\le e^{\rm app}(S_\calB^{\bsX_J},H^{\rm Kor}_{s,\alpha,\bsgamma},L_p)
 \le c_\delta\,C^{\rm sub}_{p,\epsilon}\, \Big(\frac{1}{|J|}\Big)^{\theta^{\rm sub}_{p,\epsilon} (1-\delta)}.
\end{align}  
  
\section{Error analysis} \label{sec:error}

The error for our method \eqref{eq:magic1}--\eqref{eq:magic3} with the full lattice $\bsX$ is bounded as follows:
\begin{align}\label{eq:err}
 |I(f) - Q^{\bsX}_\bspsi(f) | 
 &= \Big| \int_{[0,1]^s} \big( F(\bsy) -  (A^{\bsX}_{\rm ker}F)(\bsy)\big) \,\mu(\bsy)\,\rd\bsy \Big| \nonumber\\
 &\le \inf_{p\in [2,\infty]} \big\{ \|F-A^{\bsX}_{\rm ker}F\|_{L_p([0,1]^s)}\, \|\mu\|_{L_{p'}([0,1]^s)} 
 \big\}\nonumber\\
 &\le \inf_{p\in [2,\infty]} \big\{
 e^{\rm app}(A^{\bsX}_{\rm ker}, H^{\rm Kor}_{s,\alpha,\bsgamma}, L_p)\,
 \|F\|_{{\rm Kor},s,\alpha,\bsgamma}\, \|\mu\|_{L_{p'}([0,1]^s)} \big\} \nonumber\\
 &\le \inf_{p\in [2,\infty]} \big\{
 e^{\rm app}(A^{\bsX}_{\rm ker}, H^{\rm Kor}_{s,\alpha,\bsgamma}, L_p)\,
 \|\mu\|_{L_{p'}([0,1]^s)} \big\}\,\|f\|_{{\rm Sob},s,\alpha,\bsbeta,\omega},
\end{align}
where we first use H\"older's inequality with the conjugate pair $1/p + 1/p' = 1$, and then apply the worst case $L_p$ approximation error bound assuming that $F\in H^{\rm Kor}_{s,\alpha,\bsgamma}$, and finally we connect the Korobov norm of the transformed $F$ to the Sobolev norm of the original integrand $f$ with rescaled weights $\bsbeta$ and a weight function $\omega$, see Lemma~\ref{lem:norm-link}.
An analogous bound holds when the full lattice $\bsX$ is replaced by a subsampled lattice $\bsX_J$.

With $1/p+1/p'=1$, we have $\|\mu\|_{L_{p'}([0,1]^s)} = \xi_p^s$, where
\begin{align} \label{eq:xi-p}
  \xi_p 
  &\coloneqq \frac{\pi}{2} \bigg(\int_0^1 |\sin(2\pi y)|^{p'} \,\rd y\bigg)^{1/p'} 
  =  \frac{\pi}{2} \bigg(
 \frac{\Gamma(\frac{p'+1}{2})}{\sqrt{\pi}\,\Gamma(\frac{p'+2}{2})} \bigg)^{1/p'} 
  = \begin{cases}
   \frac{\pi}{\sqrt{8}} \le 1.110721 & \mbox{if } p = 2, \\
   1 & \mbox {if } p = \infty,
 \end{cases}
\end{align}
and $\Gamma(\cdot)$ is the Gamma function. For example, we have also $\xi_{2.5} \approx 1.07999$, $\xi_3 \approx 1.062643$, $\xi_{10} \approx 1.01554$, $\xi_{20} \approx 1.0075$. 

\subsection{Main theorem}

We can now substitute the $L_p$ approximation error bounds from Theorem~\ref{thm:combined} into \eqref{eq:err} to arrive at our main theorem.

\begin{theorem} \label{thm:main}
Let the Korobov space $H^{\rm Kor}_{s,\alpha,\bsgamma}$ and Sobolev space $H^{\rm Sob}_{s,\alpha,\bsbeta,\omega}$ with integer $\alpha\ge 1$ be related by their weights as defined in \eqref{eq:gamma_scaled}. For $N\ge 7$ sufficiently large, obtain $\bsz\in\bbZ^s$ by the component-by-component construction in \cite[Algorithm~3.2]{KMN25} so that
$\calS_N(\bsz) = \calS_{s,\alpha,\bsgamma,N}(\bsz)\le 1/4$ satisfies \eqref{eq:S-bound}. Define the frequency index set $\calB$ as in \eqref{eq:def-AB}.
Then, for all $f\in H^{\rm Sob}_{s,\alpha,\bsbeta,\omega}$ and arbitrarily small $\epsilon>0$, we have the following results:
\begin{itemize}
\item [\textnormal{(a)}]
The method \eqref{eq:magic1}--\eqref{eq:magic3} with the full lattice $\bsX = \{\frac{k}{N}\bsz\bmod 1\}_{k=0}^{N-1}$ satisfies
\begin{align} \label{eq:A-final} 
  |I(f) - Q^{\bsX}_\bspsi(f)|  
  &\le \inf_{p\in [2,\infty]} \Big\{
  C^{\rm full}_{p,\epsilon}\,\Big(\frac{1}{\varphi(N)}\Big)^{\theta^{\rm full}_{p,\epsilon}}\,
  \xi_p^s \Big\}\,\|f\|_{{\rm Sob},s,\alpha,\bsbeta,\omega},
  \nonumber\\
  \theta^{\rm full}_{p,\epsilon}
  &\to \tfrac{\alpha}{2}-\tfrac{1}{4}+\tfrac{1}{2p} \quad\mbox{as}\quad \varepsilon\to 0.
\end{align}

\item [\textnormal{(b)}]
For $t\ge 4$, let $J\subseteq\{0, \dots, N-1\}$ be a multiset of uniformly i.i.d.\ drawn integers with
$|J| \coloneqq \lceil 12\,|\calB|\,(\log|\calB| + t) \rceil$.
Then the method \eqref{eq:magic1}--\eqref{eq:magic3} with the subsampled lattice $\bsX_J = \{\frac{k}{N}\bsz\bmod 1\}_{k\in J}$ satisfies, with probability exceeding $1-3\exp(-t)$,
\begin{align} \label{eq:B-final}
  |I(f) - Q^{\bsX_J}_\bspsi(f)|  
  &\le \inf_{p\in [2,\infty]} \Big\{
  C^{\rm sub}_{p,\epsilon}\,\Big(\frac{\log|J|+t}{|J|}\Big)^{\theta^{\rm sub}_{p,\epsilon}}\,
  \xi_p^s \Big\}\,\|f\|_{{\rm Sob},s,\alpha,\bsbeta,\omega}, 
  \nonumber\\
  \theta^{\rm sub}_{p,\epsilon}
  & \to \alpha-\tfrac{1}{2}+ \tfrac{1}{p} \quad\mbox{as}\quad \varepsilon\to 0.
\end{align}  

\end{itemize}
All the constants $C^{\rm full}_{p,\epsilon}$ and $C^{\rm sub}_{p,\epsilon}$ are bounded independently of $s$ if the weights $\bsgamma$ satisfy \eqref{eq:tract}.
\end{theorem}
 
Hence, for the case of the subsampled lattice, at the two extremes we have close to the rate $|J|^{-(\alpha-1/2)}$ for $p=\infty$ and $|J|^{-\alpha}\,(1.110721)^s$ for $p=2$. Our result holds for the infimum over all values of $p\in [2,\infty]$, with a trade-off between the polynomial rate in $|J|$ and the exponential dependence on $s$. We stress that our method does not depend on the value of $p$ and we do not need to specify a value of $p$ for the error bound to hold. We will automatically get the smallest error bound over all possible values of~$p$. 
 
\subsection{Doubling the rate} \label{sec:double}

In \cite{VS25} a doubling of the rate effect was proven when the function has \emph{twice} the smoothness of the Hilbert space corresponding to the kernel, with no need to change the points or the method.
Independently, this effect was investigated in \cite{KSW25,Wenzel26} with the focus on classical Sobolev spaces, where also intermediate spaces (i.e., smoothness increased by a factor between $1$ and $2$) were investigated and the sufficient condition for the increased rate effect was proven to be also necessary.
Specializing to the weighted Korobov space which is embedded in $L_2([0,1]^s)$, the crucial property is, for $\sigma\in[1,2]$, 
 $F \in H^{\rm Kor}_{s,\sigma\alpha,\bsgamma^\sigma}$, and
 $G \in H^{\rm Kor}_{s,(2-\sigma)\alpha,\bsgamma^{2-\sigma}}$, 
 
\begin{align}
     |\langle F,G\rangle_{{\rm Kor},s,\alpha,\bsgamma}|
     &= \sum_{\bsh\in\mathbb Z^s} 
     (r_{s,\alpha,\bsgamma}(\bsh))^{\sigma/2}\, |\widehat F_{\bsh}|\, 
     (r_{s,\alpha,\bsgamma}(\bsh))^{(2-\sigma)/2}\, |\widehat G_{\bsh}| \nonumber\\
     &= \sum_{\bsh\in\mathbb Z^s} 
     (r_{s,\sigma\alpha,\bsgamma^\sigma}(\bsh))^{1/2}\, |\widehat F_{\bsh}|\, 
     (r_{s,(2-\sigma)\alpha,\bsgamma^{2-\sigma}}(\bsh))^{1/2}\, |\widehat G_{\bsh}| \nonumber\\
     &\le \|F\|_{{\rm Kor},s,\sigma\alpha,\bsgamma^\sigma}\, 
     \|G\|_{{\rm Kor},s,(2-\sigma)\alpha,\bsgamma^{2-\sigma}} \,.\label{eq:crucial}
\end{align}

In the following theorem we state the doubling of the rate effect and extend it on the one hand to intermediate spaces $\sigma\in [1,2]$ using techniques from \cite{Wenzel26}, and on the other hand to general $L_p$ approximation for $p\in [1,\infty]$.
Note that changing weights in a weighted Korobov space does not change the membership of the space; it only changes the size of the unit ball.

The advantage compared with replacing the kernel of smoothness~$\alpha$ by the kernel of smoothness~$\sigma\alpha$ is that a lower smoothness method is often more numerically stable and that lower order kernels are easier to compute.

\begin{theorem}\label{thm:double} 
Let $A^\bsX_{\rm ker}$ be defined by the kernel of the weighted Korobov space $H^{\rm Kor}_{s,\alpha,\bsgamma}$ with smoothness $\alpha$ and weights $\bsgamma$.
    For all $p\in [1,\infty]$, $\sigma\in[1,2]$, and $F \in H^{\rm Kor}_{s,\sigma\alpha,\bsgamma^\sigma}$ we have 

\begin{align} \label{eq:double}
  \|F-A^\bsX_{\rm ker}F\|_{L_p([0,1]^s)} 
  \le e^{\rm app}(A^\bsX_{\rm ker},H^{\rm Kor}_{s,\alpha,\bsgamma},L_p)\,
  \Big(e^{\rm app}(A^\bsX_{\rm ker},H^{\rm Kor}_{s,\alpha,\bsgamma},L_2)\Big)^{\sigma-1}\,
  \|F\|_{{\rm Kor},s,\sigma\alpha,\bsgamma^\sigma}.
\end{align}

    Equivalently, for all $p\in [1,\infty]$ and $\sigma\in[1,2]$ we have 
\begin{align*}
  e^{\rm app}(A^\bsX_{\rm ker},H^{\rm Kor}_{s,\sigma\alpha,\bsgamma^\sigma},L_p)
  \le e^{\rm app}(A^\bsX_{\rm ker},H^{\rm Kor}_{s,\alpha,\bsgamma},L_p)\,
  \Big(e^{\rm app}(A^\bsX_{\rm ker},H^{\rm Kor}_{s,\alpha,\bsgamma},L_2)\Big)^{\sigma-1}.
\end{align*}
Analogous results hold if the full lattice $\bsX$ is replaced by a subsampled lattice $\bsX_J$.
\end{theorem}

\begin{proof} 
Throughout the proof let $F\in H_{s,\sigma\alpha,\bsgamma^\sigma}^{\rm Kor}$ and $ G\in H_{s,(2-\sigma)\alpha,\bsgamma^{2-\sigma}}^{\rm Kor}$.
Using Hölder's inequality with parameters $1/(2-\sigma)$ and $1/(\sigma-1)$ we have
\begin{align}
    \|G\|_{{\rm Kor}, s, (2-\sigma)\alpha, \bsgamma^{2-\sigma}}
    &= \bigg( \sum_{\bsh\in\mathbb Z^s} (r_{s,\alpha,\bsgamma}(\bsh))^{2-\sigma} |\widehat G|^{4-2\sigma} |\widehat G|^{2\sigma-2} \bigg)^{1/2} \nonumber\\
    &\le \bigg(\bigg( \sum_{\bsh\in\mathbb Z^s} r_{s,\alpha,\bsgamma}(\bsh)|\widehat G|^2 \bigg)^{2-\sigma}
    \bigg( \sum_{\bsh\in\mathbb Z^s} |\widehat G|^2 \bigg)^{\sigma-1} \bigg)^{1/2} \nonumber\\
    &= \|G\|_{{\rm Kor},s,\alpha,\bsgamma}^{2-\sigma} \|G\|_{L_2([0,1]^s)}^{\sigma-1} \,. \label{eq:sth}
\end{align}
Now we use  the fact that $A^\bsX_{\rm ker}$ is an orthogonal projection from  $H^{\rm Kor}_{s,\alpha,\bsgamma}$ onto 
$\mathrm{span}\{ K_{s,\alpha, \bsgamma}^{\rm Kor}(\bst_k, \cdot): k = 0,\ldots,  N-1\}$, which follows from the reproducing property of the kernel and the interpolation property.    Using also 
 \eqref{eq:crucial} and \eqref{eq:sth}, we then  obtain
\begin{align*}
    \|F-A_{\rm ker}^\bsX F\|_{{\rm Kor}, s,\alpha,\bsgamma}^2
    &= \langle F, F-A_{\rm ker}^\bsX F\rangle_{{\rm Kor},s,\alpha,\bsgamma} \\
    &\le \|F\|_{{\rm Kor}, s,\sigma\alpha,\bsgamma^\sigma}\, 
    \|F-A_{\rm ker}^\bsX F\|_{{\rm Kor}, s,(2-\sigma)\alpha,\bsgamma^{2-\sigma}} \\
    &\le \|F\|_{{\rm Kor}, s,\sigma\alpha,\bsgamma^\sigma}\,
    \|F-A_{\rm ker}^\bsX F\|_{{\rm Kor}, s,\alpha,\bsgamma}^{2-\sigma}
    \|F-A_{\rm ker}^\bsX F\|_{L_2([0,1]^s)}^{\sigma-1} \,.
\end{align*}
Canceling the common factor on both sides gives
\begin{equation*}
    \|F-A_{\rm ker}^\bsX F\|_{{\rm Kor}, s,\alpha,\bsgamma}
    \le \|F\|_{{\rm Kor}, s,\sigma\alpha,\bsgamma^\sigma}^{1/\sigma}\,
    \|F-A_{\rm ker}^\bsX F\|_{L_2([0,1]^s)}^{(\sigma-1)/\sigma} \,.
\end{equation*}
For all $p\in [1,\infty]$ this yields
\begin{align}
     \|F-A^\bsX_{\rm ker}F\|_{L_p([0,1]^s)}
     &= \| (F-A^\bsX_{\rm ker}F) - A^\bsX_{\rm ker}(F-A^\bsX_{\rm ker}F)\|_{L_p([0,1]^s)} \nonumber\\
     &\le e^{\rm app}(A^\bsX_{\rm ker},H^{\rm Kor}_{s,\alpha,\bsgamma},L_p)\,
      \|F-A^\bsX_{\rm ker}F\|_{{\rm Kor},s,\alpha,\bsgamma} \nonumber\\ 
     &\le e^{\rm app}(A^\bsX_{\rm ker},H^{\rm Kor}_{s,\alpha,\bsgamma},L_p)\,
     \|F\|_{{\rm Kor}, s,\sigma\alpha,\bsgamma^\sigma}^{1/\sigma}\,
     \|F-A_{\rm ker}^\bsX F\|_{L_2([0,1]^s)}^{(\sigma-1)/\sigma} \label{eq:long}\,.
\end{align}
Taking $p=2$ in \eqref{eq:long} and canceling a common factor 
gives $\|F-A^\bsX_{\rm ker}F\|_{L_2([0,1]^s)}^{1/\sigma} \le e^{\rm app}(A^\bsX_{\rm ker},H^{\rm Kor}_{s,\alpha,\bsgamma},L_2)\,\|F\|_{{\rm Kor},s,\sigma\alpha,\bsgamma^\sigma}^{1/\sigma}$.
Substituting this back into \eqref{eq:long} for general $p$ gives \eqref{eq:double}.
\end{proof}

We can now extend Theorem~\ref{thm:main} to the case where $f$ is smoother compared to the reproducing kernel used in the method.
For the sake of simplicity we only state the result with double the smoothness, i.e, $\sigma=2$.

\begin{corollary} \label{cor:double}
Consider the Korobov space $H^{\rm Kor}_{s,\alpha,\bsgamma}$ and Sobolev space $H^{\rm Sob}_{s,2\alpha,\bseta^2,\omega}$ with integer $\alpha\ge 1$ and weights related by
\begin{align*}
  \gamma_\setu^2 \coloneqq c_{2\alpha}^{|\setu|} \, \eta_\setu^2, 
  \quad\mbox{with}\quad
  c_{2\alpha} = \sum_{\lambda=1}^{2\alpha} S(2\alpha,\lambda)^2.
\end{align*}
Let $A^\bsX_{\rm ker}$ be defined by the kernel of $H^{\rm Kor}_{s,\alpha,\bsgamma}$. Suppose that $f\in H^{\rm Sob}_{s,2\alpha,\bseta^2,\omega}$. Under the same remaining setup as Theorem~\ref{thm:main}, the error bound \eqref{eq:A-final} for the full lattice is replaced by
\begin{align} \label{eq:A-double} 
  |I(f) - Q^{\bsX}_\bspsi(f)|  
  &\le \inf_{p\in [2,\infty]} \Big\{
  C^{\rm full}_{p,\epsilon}\,C^{\rm full}_{2,\epsilon}\,
  \Big(\frac{1}{\varphi(N)}\Big)^{\theta^{\rm full}_{p,\epsilon} + \theta^{\rm full}_{2,\epsilon}}\,
  \xi_p^s \Big\}\,\|f\|_{{\rm Sob},s,2\alpha,\bseta^2,\omega},
  \nonumber\\
  \theta^{\rm full}_{p,\epsilon} +\theta^{\rm full}_{2,\epsilon} 
  &\to \alpha-\tfrac{1}{4} + \tfrac{1}{2p} \quad\mbox{as}\quad \varepsilon\to 0,
\end{align}
while the error bound \eqref{eq:B-final} for the subsampled lattice is replaced by
\begin{align} \label{eq:B-double}
  |I(f) - Q^{\bsX_J}_\bspsi(f)|  
  &\le \inf_{p\in [2,\infty]} \Big\{
  C^{\rm sub}_{p,\epsilon}\,C^{\rm sub}_{2,\epsilon}\,
  \Big(\frac{\log|J|+t}{|J|}\Big)^{\theta^{\rm sub}_{p,\epsilon} + \theta^{\rm sub}_{2,\epsilon}}\,
  \xi_p^s \Big\}\,\|f\|_{{\rm Sob},s,2\alpha,\bseta^2,\omega}, 
  \nonumber\\
  \theta^{\rm sub}_{p,\epsilon} +\theta^{\rm sub}_{2,\epsilon}
  &\to 2\alpha-\tfrac{1}{2}+ \tfrac{1}{p} \quad\mbox{as}\quad\varepsilon\to 0.
\end{align}  
\end{corollary}

\begin{proof}
Applying \eqref{eq:double} with $\sigma = 2$ after the first inequality in \eqref{eq:err} gives 
\begin{align*}
 |I(f) - Q^{\bsX}_\bspsi(f) | 
 \le \inf_{p\in [2,\infty]} \big\{
 e^{\rm app}(A^{\bsX}_{\rm ker}, H^{\rm Kor}_{s,\alpha,\bsgamma}, L_p)\,
 e^{\rm app}(A^{\bsX}_{\rm ker}, H^{\rm Kor}_{s,\alpha,\bsgamma}, L_2)\, \xi_p^s \big\}\,
 \|F\|_{{\rm Kor},s,2\alpha,\bsgamma^2}.
\end{align*}
Using Lemma~\ref{lem:norm-link} with $\alpha$ and $\gamma_\setu$ replaced by $2\alpha$ and $\gamma_\setu^2$, respectively, we obtain $\|F\|_{{\rm Kor},s,2\alpha,\bsgamma^2} \le \|f\|_{{\rm Sob},s,2\alpha,\bseta^2,\omega}$, where the sequence $\bseta = (\eta_\setu)_{\setu\subset\bbN}$ is defined by $\gamma_\setu^2 \,\coloneqq c_{2\alpha}^{|\setu|}\,\eta_\setu^2$.
The result now follows by applying the $L_p$ approximation error bounds from Theorem~\ref{thm:combined}.
\end{proof}

\section{Numerical experiments} \label{sec:num}

In this section we present results for the weighted quadrature rule for both the full and subsampled lattice for a range of values $\alpha \in \{1, 2, 3\}$ in dimensions up to $s = 64$.
The values of $N$ tested are prime numbers (roughly doubling at each step) in the range
\[
N \in \{ 131,
         257,
         521,
        1031,
        2053,
        4099,
        8209,
       16411,
       32771,
       65537,
      131101,
      262147,
      524309
      \}
\]
For the full lattice, the weights are computed using the FFT via \eqref{eq:w_fft} as described in Section~\ref{sec:imp-full} and for the subsampled lattice the system \eqref{eq:w_J} is solved using the conjugate gradient iterative method with FFT (Method 3) as described in Section~\ref{sec:imp-sub}.  The subsampling set $J$ is chosen randomly (with replacement) with size $|J| = \lceil \sqrt{N} \log N \rceil$, as in \cite{BGKS25}.
The generating vector is constructed using the component-by-component construction for approximation \cite{GS26+} using POD weights given by
\begin{equation}
\label{eq:pod_num}
\gamma_\setu  = \Gamma_{|\setu|} \prod_{j \in \setu} \gamma_j 
\quad \text{with} \quad
\gamma_j = \frac{1}{j^q}, \quad \Gamma_\ell = \ell!\,,
\end{equation}
where $q > 1$ is a parameter matched to the decay of the variables in each test integrand below.

All computations were performed on the computational cluster Katana \cite{katana} at UNSW Sydney.

\subsection{Example 1: reciprocal function}
First, we test our method for the reciprocal function, for $ q \geq 2$ given by
\begin{align} \label{eq:func-reciprocal}
 f(\bsx) = \frac{1}{1 + \sum_{j = 1}^s (x_j - \frac{1}{2})/j^q},
\end{align}
which has been studied previously in, e.g., \cite{GanSch16,BGKS25}. We study the convergence of the new quadrature rule applied to $f$ for two different decays $ q \in\{2.5, 6\}$ and dimension $s \in\{ 2, 4, 64\}$.

\begin{table}[!t]
\centering
\begin{tabular}{c||c|c}
$s$ & $I(f)$ with $q=2.5$ & $I(f)$ with $q = 6$ \\
\hline\hline
$2$ & 1.1032908827 & 1.0986484606\\
$4$ & 1.1040739938 & 1.0986487483 \\
$64$ & 1.1041644920 & 1.0986487490
\end{tabular}
\caption{Computed reference value for integral of reciprocal function \eqref{eq:func-reciprocal}.}
\label{tab:f_recip_int}
\end{table}

The error is computed by comparing the weighted quadrature rules \eqref{eq:quad2} and \eqref{eq:quad_sub} to a reference integral, which was computed to accuracy $< 10^{-10}$ using a combination of explicitly integrating the first two dominant dimensions then applying a randomly-shifted rank-1 lattice rule with $2^{23}$ points and $64$ random shifts ($2^{29}$ points in total). The computed reference value of the integral is given in Table~\ref{tab:f_recip_int}. Since the first two dimensions are integrated analytically, the value of the integral is exact for $s = 2$. 


In Figure~\ref{fig:ex1:q=6} we plot the error for the reciprocal function with the faster decay $q = 6$ (more smooth). 
Subfigures (a), (b), (c) give results in different dimensions for $\alpha = 1, 2, 3$, respectively, and subfigure (d) compares the full and subsampled lattice for different $\alpha$ in dimension $64$. The two dashed lines indicate two rates we expect from our theory as summarised in Table~\ref{tab:conv}: green is the best possible rate with dimension dependence ($p = 2$) and doubling the rate effect, whereas magenta is the slower dimension independent rate ($p = \infty$) without the doubling the rate effect. Note that the number of function evaluations on the horizontal scale is not the same across the figures. For this faster decay (i.e., more smooth) case, we observe the best possible rate $N^{-\alpha}$ for the full lattice for all values of $\alpha$, with even faster convergence for $\alpha = 1$ or in lower dimensions for $\alpha = 2, 3$. Comparing the full lattice with the subsampled lattice for dimension $s = 64$ in Figure~\ref{fig:ex1:q=6}(d), we observe that the full and subsampled lattices exhibit similar behaviour, with both the full and subsampled lattice converging at a rate of roughly $\alpha$.
 

\begin{figure}[t]
\centering
\begin{subfigure}{0.49\textwidth}
\includegraphics[scale=0.85]{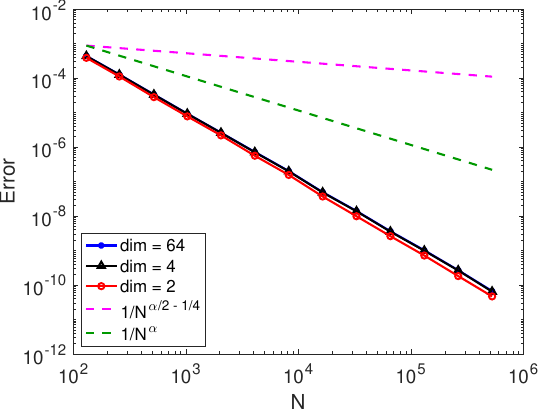}
\caption{Full lattice for $\alpha = 1$ and dimensions $s = 2, 4, 64$.}
\end{subfigure}
\begin{subfigure}{0.49\textwidth}
\includegraphics[scale=0.85]{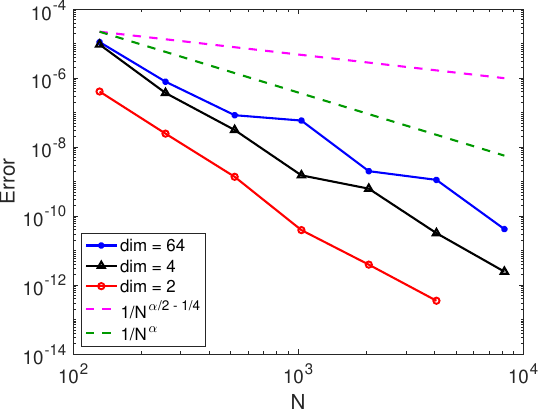}
\caption{Full lattice for $\alpha = 2$ and dimensions $s = 2, 4, 64$.}
\end{subfigure}
\begin{subfigure}{0.49\textwidth}
\includegraphics[scale=0.85]{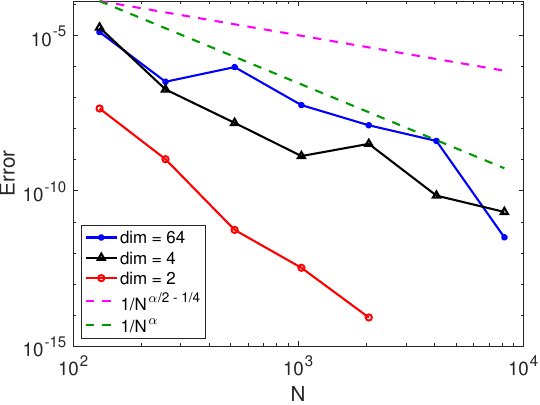}
\caption{Full lattice for $\alpha = 3$ and dimensions $s = 2, 4, 64$.}
\end{subfigure}
\begin{subfigure}{0.49\textwidth}
\includegraphics[scale=0.85]{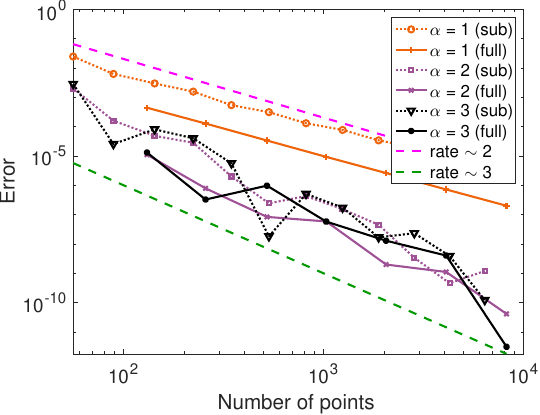}
\caption{Full vs subsampled lattice for $\alpha = 1,2, 3$ ($s = 64$).}
\end{subfigure}
\caption{Convergence of error for reciprocal function \eqref{eq:func-reciprocal} for decay $q = 6$.}
\label{fig:ex1:q=6}
\end{figure}


In Figure~\ref{fig:ex1:q=2.5} we plot the results for the slower decay $q = 2.5$ (less smooth), where again subfigures (a), (b), (c) give results for $\alpha = 1, 2, 3$, respectively, and (d) compares the full and subsampled lattice. For this more difficult function, we only observe the best possible rate of $N^{-\alpha}$ for $\alpha = 1$ or for lower dimensions when $\alpha = 2, 3$. 
For $\alpha = 2$ in dimension 64, we only begin to observe close to the best rate of 2 for higher $N$, whereas for $\alpha = 3$ we only observe the slower rate of $\frac{\alpha}{2} - \frac{1}{4} = 1.25$, i.e., we do not observe the doubling the rate effect.
Similar to the $q = 6$ case, the full and subsampled rules exhibit comparable behavior when using a comparable number of points overall.

\begin{figure}[t]
\centering
\begin{subfigure}{0.49\textwidth}
\includegraphics[scale=0.85]{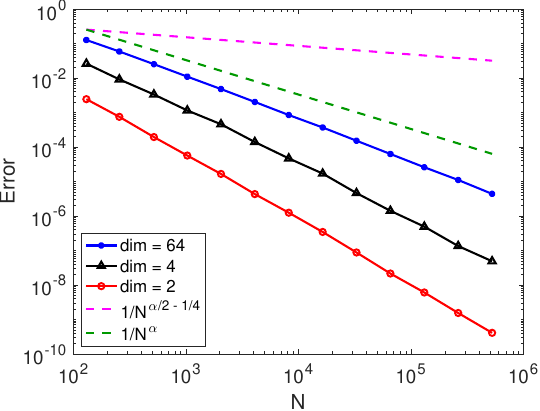}
\caption{Full lattice for $\alpha = 1$ and dimensions $s = 2, 4, 64$.}
\end{subfigure}
\begin{subfigure}{0.49\textwidth}
\includegraphics[scale=0.85]{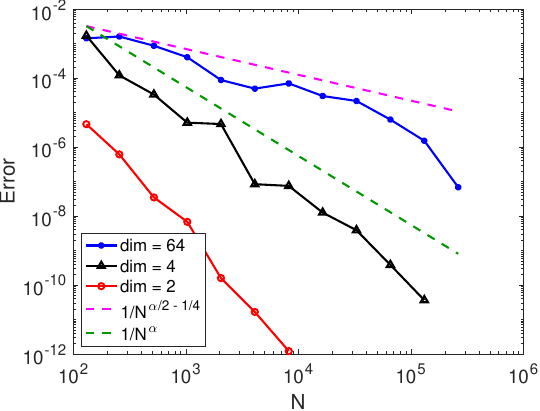}
\caption{Full lattice for $\alpha = 2$ and dimensions $s = 2, 4, 64$.}
\end{subfigure}
\begin{subfigure}{0.49\textwidth}
\includegraphics[scale=0.85]{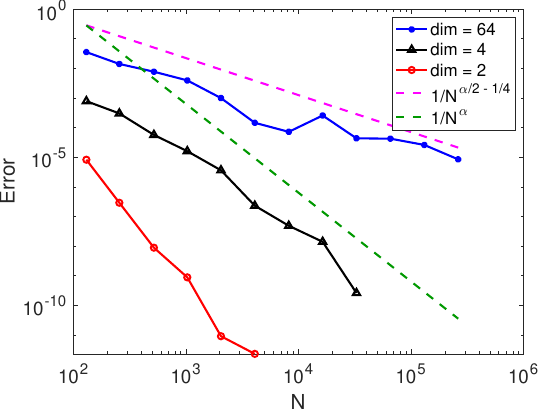}
\caption{Full lattice for $\alpha = 3$ and dimensions $s = 2, 4, 64$.}
\end{subfigure}
\begin{subfigure}{0.49\textwidth}
\includegraphics[scale=0.85]{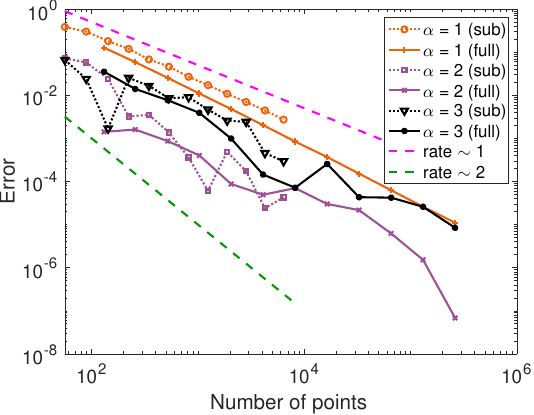}
\caption{Full vs subsampled lattice for $\alpha = 1,2, 3$ ($s = 64$).}
\end{subfigure}
\caption{Convergence of error for reciprocal function \eqref{eq:func-reciprocal} for decay $q = 2.5$.}
\label{fig:ex1:q=2.5}
\end{figure}

\subsection{Example 2: product function}
For a given $\alpha \geq 1$ and $q \geq 1$, we also consider the product function 
\begin{align} \label{eq:func-product}
f(\bsx) = \prod_{j = 1}^s \frac{1 + x_j^{\alpha + 1/2}/j^q}{1 + 1/[(\alpha + 3/2)j^q]}
\end{align}
for which the integral is equal to 1 in all dimensions. Note that for a given $\alpha$, we have that $f \in H^\mathrm{Sob}_{s, \alpha, \bsbeta, \omega}$, but $f$ does not belong to $H^\mathrm{Sob}_{s, \alpha + 1, \bsbeta, \omega}$, and so we do not expect to see the full doubling the rate effect.

For this product function we again tested two decays $q \in \{2.5, 6\}$ and $\alpha \in \{1, 2, 3\}$.
The results for $q = 6$ are given in Figure~\ref{fig:ex2:q=6}. We observe similar behavior as for the reciprocal function, however, since the product function is not in $H^\mathrm{Sob}_{s, \alpha + 1, \bsbeta, \omega}$, for higher $\alpha$ and decay ($\alpha = 3$, $q = 6$ in Figure~\ref{fig:ex2:q=6}(c)) we do not observe the full doubling the rate effect, but instead see an intermediate rate as predicted by Theorem~\ref{thm:double}. The results for $q = 2.5$ are similar to the reciprocal function and so have not been included.

\begin{figure}[t]
\centering
\begin{subfigure}{0.49\textwidth}
\includegraphics[scale=0.85]{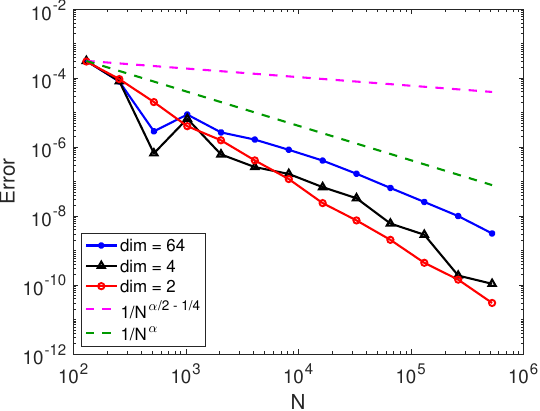}
\caption{Full lattice for $\alpha = 1$ and dimensions $s = 2, 4, 64$.}
\end{subfigure}
\begin{subfigure}{0.49\textwidth}
\includegraphics[scale=0.85]{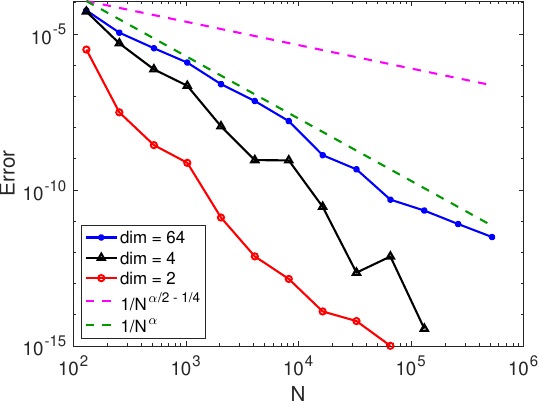}
\caption{Full lattice for $\alpha = 2$ and dimensions $s = 2, 4, 64$.}
\end{subfigure}
\begin{subfigure}{0.49\textwidth}
\includegraphics[scale=0.85]{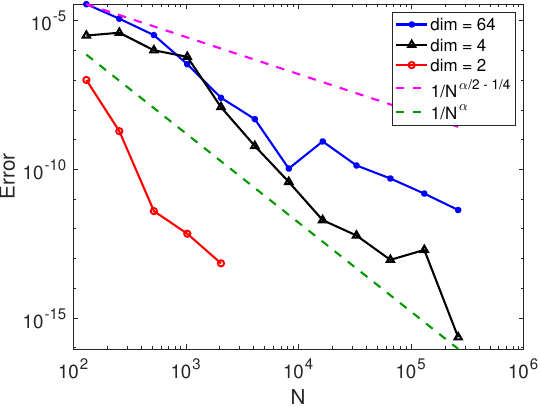}
\caption{Full lattice for $\alpha = 3$ and dimensions $s = 2, 4, 64$.}
\end{subfigure}
\begin{subfigure}{0.49\textwidth}
\includegraphics[scale=0.85]{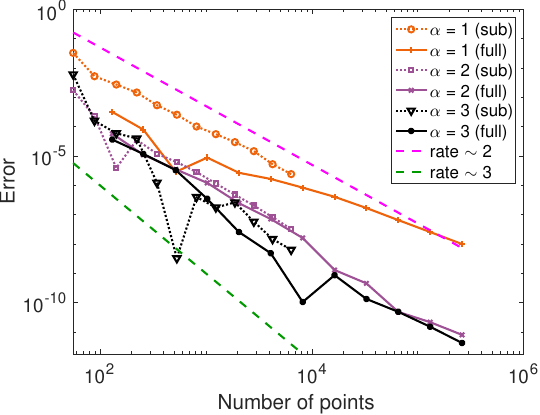}
\caption{Full vs subsampled lattice for $\alpha = 1,2, 3$ ($s = 64$).}
\end{subfigure}
\caption{Convergence of error for product function \eqref{eq:func-product} for decay $q = 6$.}
\label{fig:ex2:q=6}
\end{figure}


\section{Other related strategies} \label{sec:other}

\subsection{A common periodization method} \label{sec:periodization}

Many papers (see for example \cite{Kor63} and \cite{Sid93}) 
have used a coordinate transformation to achieve a somewhat smooth periodic integrand, as in (for the univariate case) 
\[
\int_0^1 f(x) \,\rd x = \int_0^1 f(\phi(y))\,  \phi'(y) \,\rd y.
\]  
Here $\phi$ is a smooth increasing function mapping $[0,1]$ onto $[0,1]$, with one or more vanishing derivatives at both $0$ and $1$. Examples are $\phi(t) \coloneqq 3t^2 - 2t^3$ due to Korobov \cite{Kor63} and $\phi(t) \coloneqq t - \sin(2\pi t)/(2\pi)$ due to Sidi \cite{Sid93}.  Because of the vanishing derivatives at the end points the resulting integrand  $f(\phi(y))\, \phi'(y)$ can be considered as a single period of a somewhat smooth periodic function on $\mathbb{R}$, whose integral  can then be approximated by the trapeziodal rule, which from the Euler--Maclaurin expansion  can be highly acccurate for integration over a full period of a periodic function.  Since periodicity is enforced by requiring  $\phi'(0) = 0 = \phi'(1)$, the end result is a rule of the form   
\[
\int_0^1 f(x) \,\rd x \approx \frac{1}{N}\sum_{k =1 }^{N-1} f(\phi(t_k))\,\phi'(t_k), \quad \mbox{with } t_k = \tfrac{k}{N} \mbox{ for } k = 1,\ldots, N-1.
\]  

Note that that this commonly used strategy (even with its lattice generalization to the multivariate case) is completely different from the method introduced in this paper.  Firstly, our function $\psi$ in \eqref{eq:psi} is \emph{not} monotone.  To the contrary, our function $\psi$ traverses the interval $[0,1]$ twice with respect to $x$ (see Figure~\ref{fig:trans}), first increasing from $0$ to $1$, then smoothly reversing. Secondly, our transformed  integrand $f(\psi(y))\, \mu(y)$ makes no pretence to be smooth --- rather, the weight function $\mu(y) = |\psi'(y)|/2$ has a simple discontinuity in its first derivative.

Multivariate applications of Korobov's periodization technique using lattice points have been considered since the paper \cite{Kor63}, but mostly not giving an error analysis in the setting of weighted mixed derivative Hilbert spaces,  as in the current paper. An exception is \cite{KSW07}, whose main conclusion was that this periodization strategy may fail in high dimensions, in that the error can grow exponentially with $s$ because the Sobolev space norm of the periodized integrand $f(\phi(y))\, \phi'(y)$ can be exponentially large compared to the Sobolev space norm of the original function $f$. Finally, the paper \cite{BH1992} explores a number of periodizing transformations in the multivariate-plus-lattice context, including a favoured transcendental transformation.

\subsection{The method of product integration}
\label{sec:prod_int}
The method of the present paper can be seen as a multivariate variant of the so-called method of  \emph{product integration}, see \cite{You54, EllPag76}, which was designed to  handle univariate integrals of the form $\int_0^1 f(x)\,\nu(x)\, \rd x$, where $f$ is smooth and $\nu$ is (weakly) singular.  The essence of the method is to approximate $f$ to high accuracy by polynomial interpolation at suitable points $t_1, t_2, \ldots,t_N\in [0,1]$, and then to evaluate the integral of a polynomial basis set times  $\nu$ numerically  to high accuracy. The approximate integral is  expressed finally in the form 
\[
\int_0^1 f(x)\,\nu(x)\,\rd x \approx \sum_{k = 1}^N w_k\, f(t_k),
\]
where the weights $w_k$ incorporate the influence of the weakly singular function $\nu$.
In the present method the polynomial basis is replaced by a set of translated kernel functions, and the weakly singular function $\nu$ is replaced by the weight function $\mu$.

In some  situations (for example interpolation at the Clenshaw--Curtis points or the zeros of Jacobi polynomials) the computed product-integration weights $w_k$ are known to have a very desirable stability property if $\nu$ is non-negative, namely that  of ``asymptotic positivity'' \cite{SloSmi82}, meaning that
\[
\lim_{N \to\infty} \sum_{k=1}^N w_{k} = \lim_{N\to \infty} \sum_{k=1}^N  |w_{k}|.
\]
This property implies that for $N$ sufficiently large the contribution of negative weights to the sum of absolute values  becomes vanishingly  small. The asymptotic positivity property is not known to hold for the present multivariate method.

\subsection{Integrating the tent-transformed integrand}

The method of integrating after application of the so-called \emph{tent transform}, given in the  univariate case by $x = \phi(y) \coloneqq \max\{2y, 1-2y\}$, see e.g., \cite{Hic02,GSY15,SNC16,CKNS16}, can be seen as an extreme example of the periodization method in Subsection~\ref{sec:periodization}, except that here $\phi$ is not monotone, instead (in the univariate case) traversing $[0,1]$ twice. Because this $\phi(y)$ does not have any vanishing derivatives at the end points and is not smooth, the tent transform method does not permit convergence of arbitrary order. However, it does have the advantage that the resulting density is trivially equal to~$1$, and also, as shown in \cite{Hic02}, that it can achieve the convergence order appropriate to smoothness $\alpha = 2$.   

We note that this method can also be seen as an extreme example of the product-integration method from Section~\ref{sec:prod_int}, but with a periodizing function that, unlike $\psi$ given by \eqref{eq:psi}, is very far from being smooth.

\subsection{Chebyshev lattices}
\label{sec:cheb}

We have already mentioned in Subsection~\ref{sec:Cheb-lattice} that our transformed lattice points are related to Chebyshev lattices \cite{CooPop2011}. These are typically defined for an even number of points $N$. Also, the transformation $x = \cos(\pi y)$ in \cite[Definition~1]{CooPop2011} differs from $\psi$ in \eqref{eq:psi} only to the extent of mapping  $[-1,1]$ to $[-1,1]$ instead of $[0,1]$ to $[0,1]$, as here. An alternative interpretation as described in \cite{KMNN21} is that the points are first tent-transformed and then cosine-transformed, which allows also for odd~$N$. For both odd and even $N$, there is duplication of points due to the tent transform. There will be exactly $\lfloor N/2+1\rfloor$ distinct points if $\gcd(z_j,N)=1$ for at least one component $z_j$ of the lattice generating vector $\bsz$, see \cite[Lemma~12 and Subsection~4.3]{KMNN21}.
 
\subsection{Higher-order digital nets}

A key competitor to the method presented in this paper is the method of \emph{high-order digital nets} as initiated by Dick \cite{DIC08}. These are methods of a completely different kind. They are quasi-Monte Carlo methods, that is, they are equal weight rules, with $w_k = 1/N$. For the integration problem they can achieve the full order of convergence for functions in $H_\alpha^\mathrm{mix}$ \cite{GSY17} and can be computed by a component-by-component construction \cite{GOD15}.
But order of convergence is not the only consideration. Also of importance is the implied constant, including its rate of growth  with respect to dimensionality and order $\alpha$.  
There are also reported difficulties in the  implementation  of higher-order digital nets, in that  high-precision arithmetic may be required.
A detailed comparison is beyond the scope of this paper, but should be carried out in the future.

\section*{Acknowledgments}

FB, ADG, FYK and IHS acknowledge financial support from the Australian Research Council Discovery
Project (DP240100769). MG acknowledges support from the Hausdorff
Center for Mathematics (HCM) in Bonn, funded by the Deutsche Forschungsgemeinschaft
(DFG, German Research Foundation) under Germany’s Excellence Strategy – EXC-2047/2 – 390685813.
This research includes computations using the computational cluster Katana supported by Research Technology Services at UNSW Sydney.

\bibliographystyle{abbrv}
\bibliography{references.bib}

\end{document}